\documentclass[a4paper]{amsart}
\usepackage{amsmath,amssymb}
\usepackage{graphicx,subfigure}
\usepackage{mathtools}
\usepackage[hidelinks]{hyperref}

\usepackage{color}

\newtheorem{theorem}{Theorem}[section]
\newtheorem{proposition}[theorem]{Proposition}

\newtheorem{lemma}[theorem]{Lemma}

\theoremstyle{definition}

\newtheorem*{definition*}{Definition}

\theoremstyle{remark}
\newtheorem{remark}[theorem]{Remark}

\numberwithin{equation}{section}

\newcommand{\xd}{\mathsf{x}}

\newcommand{\bnu}{{\boldsymbol{\nu}}}

\newcommand{\nn}{\mbox{\boldmath$n$}}
\newcommand{\uu}{\mbox{\boldmath$u$}}
\newcommand{\vv}{\mbox{\boldmath$v$}}

\def\RR{\mathbb{R}}

\newcommand{\cB}{{\mathcal B}}

\newcommand{\QQ}{{\mathcal Q}}

\newcommand{\fp}{\mathfrak{p}}
\newcommand{\fq}{\mathfrak{q}}
\newcommand{\fr}{\mathfrak{r}}
\newcommand{\ft}{\mathfrak{t}}

\newcommand\minus\backslash

\newcommand{\abs}[1]{\left|#1\right|}

\newcommand\lan\langle
\newcommand\ran\rangle

\DeclareMathOperator\Div{div} \DeclareMathOperator\Ric{Ric}

\DeclareMathOperator{\di}{div}
\DeclareMathOperator{\gr}{grad}
\DeclareMathOperator{\cu}{curl}

\renewcommand\leq\leqslant
\renewcommand\geq\geqslant
\newlength{\intwidth}

 \DeclareMathOperator\curl{curl}

\begin{document}

\title[A symmetry theorem for Euler flows]{A symmetry theorem for localizable steady solutions of the 3D Euler equations}

\author{Daniel Peralta-Salas}
\address{Instituto de Ciencias Matem\'aticas, Consejo Superior de
  Investigaciones Cient\'\i ficas, 28049 Madrid, Spain}
\email{dperalta@icmat.es}

\author{Radu Slobodeanu}
\address{Faculty of Physics, University of Bucharest, P.O. Box Mg-11, RO--077125 Bucharest-M\u agurele, Romania}
\email{radualexandru.slobodeanu@g.unibuc.ro}

\keywords{Steady Euler flows, localizable solutions, toroidal domain, rotational symmetry}

\begin{abstract}
A steady Euler flow is localizable if the pressure function is constant along its stream lines. This property was used by Gavrilov to construct the first smooth compactly supported steady states of 3D Euler. We prove that any analytic localizable 3D Euler flow in a bounded domain $\Omega$ is axisymmetric and $\Omega$ is a rotationally symmetric domain whose transverse section is a disk or an annulus with convex boundary curves. To the best of our knowledge, this is the first symmetry theorem for 3D steady Euler flows. In the context of MHD equilibria, this result shows that Grad's conjecture holds true for magnetic fields satisfying the isodynamic condition, a property introduced by Palumbo in the 1960's to minimize the effect of particle drifts in plasma confinement devices.
\end{abstract}
\maketitle

\section{Introduction}\label{sec:intro}

The stationary Euler equations in a bounded domain $\Omega\subset\RR^3$ describe a fluid flow in equilibrium. They are formulated as
\begin{equation}\label{eq.eu}
(\uu\cdot \nabla) \uu + \nabla \Pi=0\,,\qquad \Div \uu =0\,, \text{ in } \Omega
\end{equation}
with the boundary condition $\uu\cdot \bnu=0$ on $\partial\Omega$, where $\bnu$ is the outer unit normal vector. Here $\uu$ is the velocity field of the fluid and $\Pi$ is the hydrodynamic pressure. An equivalent formulation in terms of the Bernoulli function $b:=\Pi+\frac12 |\uu|^2$,
\begin{equation} \label{eq.eu.curl}
\uu\times \curl \uu =\nabla b\,,\qquad \Div \uu =0\,,
\end{equation}
is particularly useful when studying the geometry of the solutions, the main landmark being the celebrated Arnold's structure theorem~\cite[Chapter II.1]{arn}. This formulation is also remarkable because it is analogous to the equations satisfied by an equilibrium magnetic field $B$ in the context of magnetohydrodynamics (MHD):
\begin{equation*}
B\times \curl B +\nabla P=0\,,\qquad \Div B =0\,,
\end{equation*}
where $P$ is the plasma pressure (it formally coincides with minus the Bernoulli function if $B$ plays the role of the velocity field of a fluid flow).

In 1968, when studying plasma confinement, Palumbo introduced~\cite{Pal68} a condition that he called ``isodynamic property'' of an MHD equilibrium:
\begin{equation}\label{eq:iso}
B\cdot \nabla |B|^2=0 \text{ in } \Omega\,,
\end{equation}
that is, the modulus of $B$ is constant along the magnetic lines of $B$. Palumbo observed that the condition for all first-order particle drifts to remain in a constant pressure surface is precisely the isodynamic property~\eqref{eq:iso}. This makes condition~\eqref{eq:iso} particularly relevant in order to study plasma confinement because it helps to minimize the effect of particle drifts, and therefore to diminish the harmful plasma-wall interaction in fusion reactors. However, in practice, the isodynamic condition is too strong for real-world magnetic confinement in a fusion device like a stellarator. Surprisingly enough, the same condition was rediscovered by Gavrilov in the context of the Euler equations, when he constructed smooth compactly supported steady Euler flows~\cite{Gavrilov} (see also~\cite{CV,DEP}); in fluid mechanics this is known as ``localizable'' steady states:
\begin{equation}\label{eq:loc2}
\uu \cdot \nabla \Pi  = 0 \text{ in } \Omega\,,
\end{equation}
which is equivalent to
\begin{equation}\label{eq:loc}
\uu \cdot \nabla |\uu|^2  = 0 \text{ in } \Omega\,,
\end{equation}
using the fact that the Bernoulli function $b$ is a first integral of $\uu$. It becomes evident that Equations~\eqref{eq:iso} and~\eqref{eq:loc} are the same (changing $B$ by $\uu$). It is important to emphasize that Gavrilov's construction provides a localizable axisymmetric solution, which is related to Palumbo's study~\cite{Pal68} in the context of axisymmetric MHD equilibria satisfying the isodynamic condition. However, Palumbo did not provide a complete proof of the existence of an isodynamic equilibrium, and was not interested in finding compactly supported solutions; for that one needs an additional clever trick employed by Gavrilov in his article.

Our goal in this work is to prove that any analytic localizable steady Euler flow is axisymmetric and $\Omega$ is a rotationally symmetric domain with convex section. This implies, in particular, that any compactly supported steady state that one can aim to construct using Gavrilov's trick (i.e., the localizability condition) will necessarily be axisymmetric. Palumbo and Balzano stated a weaker version of this result in the context of MHD equilibria~\cite{PB86}, under the additional assumption that $B$ is ergodic on almost all the level sets of $P$ (without studying the convexity of the poloidal section of the domain). Apparently, this work is not very well known in the plasma physics community~\cite{S95,Sch03,Al11}, and moreover, Palumbo-Balzano's proof is rather obscure and it contains several gaps and imprecisions (see Appendix~\ref{S:Balz} for a detailed discussion). Our proof has remarkable differences with respect to~\cite{PB86} and fixes all its issues.

The following is our main result. We state it in the context of steady Euler flows, but the same statement holds mutatis mutandis in the context of MHD equilibria. Notice that the localizability (or isodynamic) condition is a sort of overdetermination for the solutions of the stationary Euler equations~\eqref{eq.eu}, so it stands to reason that such solutions can only exist if they satisfy some symmetry property. To the best of our knowledge, this is the first symmetry theorem for 3D steady Euler flows; symmetry results for the 2D stationary Euler equations have attracted considerable attention in recent years, giving rise to outstanding results, see e.g.~\cite{HN19,CDG21,GPS21,HN23,Ruiz23,EFRS,EHSX}. We recall that a domain $\Omega\subset \mathbb R^3$ is rotationally symmetric if there are coordinates $(x_1,x_2,x_3)$ such that the domain remains unchanged under arbitrary rotations about the $x_3$-axis. A vector field $\uu$ in $\Omega$ is axisymmetric if it commutes with the rotation symmetry $-x_2\partial_{x_1}+x_1\partial_{x_2}$.

\begin{theorem}\label{T.main1}
Let $\uu$ be a steady Euler flow in a smooth bounded domain $\Omega$. Assume that $\uu\in C^1(\overline\Omega)\cap C^\omega(\Omega)$ is tangent to $\partial\Omega$ and satisfies the localizability condition~\eqref{eq:loc2} in $\Omega$, and that $b$ is constant on each connected component of $\partial\Omega$ and $\nabla b(x)\neq 0$ for any $x\in\partial\Omega$. Then, $\uu$ is axisymmetric and:
\begin{itemize}
\item If $\partial \Omega$ is connected, $\Omega$ is a rotationally symmetric toroidal domain with convex poloidal section.
\item Otherwise, $\Omega$ is a rotationally symmetric convex toroidal annulus (in the sense that its poloidal section is an annulus whose two boundary curves are convex).
\end{itemize}
 \end{theorem}

\begin{remark}[Grad conjecture]
A particular case of Theorem~\ref{T.main1} is when $\Omega$ is foliated by regular level sets of $b$, except for a core closed line, which implies that $\Omega$ is diffeomorphic to a solid torus. This hypothesis was introduced by Grad when studying plasma confinement properties. He conjectured that equilibria of this type must be axisymmetric or isolated~\cite{Grad}. Our theorem shows that Grad's conjecture holds true in the case of isodynamic (or localizable) analytic equilibria, which is a natural condition~\cite{Pal68} in order to improve plasma confinement.
\end{remark}

\begin{remark} In plasma physics, the notion of quasi-symmetry plays a central role~\cite{BKM20}. Recall that a vector field $\xi$ is a \textit{weak quasi-symmetry}~\cite{rod} of the  magnetic field $B$ if
$$[\xi, B]=0\,, \qquad \di \xi =0\,, \qquad \xi(\abs{B}^2)=0\,.$$
We observe that any isodynamic magnetic field $B$ is weakly quasi-symmetric because any vector field of the form $\xi = f_1(P)  B + f_2(P) \cu B$ (for arbitrary functions $f_1,f_2$) is a weak quasi-symmetry of $B$. Our result shows that an isodynamic equilibrium is, in fact, Killing symmetric.
\end{remark}

This article is organized as follows. In Section~\ref{sec:pre} we recall several constructions and identities from Riemannian geometry, and we introduce the Euler equations on Riemannian manifolds. The proof of Theorem~\ref{T.main1}, up to the convex section claim, is presented in Section~\ref{sec:proof}. It is divided in $6$ steps (Sections~\ref{SS.step1}--\ref{SS.step6}), and we use suitable local coordinates and geometric identities, this being the reason why it is crucial to write the Euler equations for arbitrary Riemannian metrics. Finally, the fact that the domain $\Omega$ has a convex poloidal section is proved in Section~\ref{S:polo}. The paper concludes with an Appendix, where we briefly discuss some of the issues in Palumbo-Balzano's proof~\cite{PB86}.

\section{Preliminary results and useful identities}\label{sec:pre}
The Euler equations can be introduced on a Riemannian 3-manifold $(M, g)$, possibly with boundary. This will be helpful for us in the next section where we have to work with non-standard coordinates in which the metric is described only by some general properties.

\subsection{Vector calculus identities on a Riemannian manifold}
Let $(M,g)$ be a Riemannian 3-manifold, where we denote the Levi-Civita connection by $\nabla$ and the volume element by $\mathrm{vol}_g$. As usual, $\flat$ stands for the isomorphism that transforms vector fields into $1$-forms, i.e., $X^\flat(Y):=g(X,Y)$ for any vector fields $X,Y$ on $M$, and $\sharp$ denotes its inverse (the isomorphism from $1$-forms into vector fields). The Hodge star operator from $p$-forms to $3-p$ forms is defined by
$$
(\ast \alpha)(X_1, \ldots , X_{3-p})\mathrm{vol}_g =\alpha \wedge X_1^\flat \wedge \ldots \wedge X_{3-p}^\flat\,,
$$
for any $p$-form $\alpha$ and any vector fields $X_i$ on $M$. The codifferential operator
$\delta=(-1)^p\ast\mathrm{d}\ast$, which takes $p$-forms and gives $(p-1)$-forms, is the formal adjoint to the exterior derivative $\mathrm{d}$, which acts on 1-forms $\alpha$ as follows
$$\mathrm{d} \alpha (X,Y)=X(\alpha(Y))-Y(\alpha(X))- \alpha([X,Y])\,,$$
for any vector fields $X,Y$. Here $[\cdot,\cdot]$ denotes the commutator (or Lie bracket) of two vector fields, which satisfies the (zero torsion) identity
\[
[X,Y]=\nabla_X Y - \nabla_Y X\,.
\]
Given any (local) vector field $X$ on $M$, we define the divergence
$$\Div X=-\delta X^\flat$$
and the curl operator
$$\curl X=\left(\ast\mathrm{d} X^\flat\right)^\sharp\,,$$
and for any function $f$, its gradient is $\nabla f = (\mathrm{d}f)^\sharp$. Given two vector fields, the cross product is defined as
$$X \times Y =\left(\ast(X^\flat \wedge Y^\flat)\right)^\sharp\,,$$
and the triple product has the property:
\begin{equation}\label{triple}
X \times (Y \times Z)= g(X, Z) Y- g(X, Y)Z\,,
\end{equation}
where $g(X,Y)$ denotes the scalar product of two vector fields $X$ and $Y$. For simplicity, we shall denote the pointwise norm of a vector field $X$ by
$$|X|^2=g(X,X)\,.$$

The expressions in coordinates of all these operators, as well as the vector calculus identities analogues in this general context, are standard; for a concise account on the subject, see e.g.~\cite{Lee}. In particular, we shall make extensive use of the following identities:
\begin{equation}\label{ident.cross}
\begin{split}
&\di (X \times Y)=g(\cu X, Y) - g(X, \cu Y),\\
&\curl (X \times Y) = (\di Y)X - (\di X)Y - [X,Y]\,.
\end{split}
\end{equation}
Finally, let us recall that $X_x \in T_x M$ represents the vector field $X$ evaluated at the point $x\in M$ and that $X(f)$ denotes the action of $X$ (understood as a first order differential operator) on the function $f$:
\[
X(f)=g(X,\nabla f)\,
\]

\subsection{The Euler equations on a Riemannian manifold}
The stationary Euler equations~\eqref{eq.eu} on $(M,g)$ read as
\begin{equation*}
\nabla_{\uu} \uu + \nabla \Pi=0\,,\qquad \Div \uu =0\,,
\end{equation*}
where  $\uu$ is a vector field on $M$ (tangent to $\partial M$) and $\Pi$ is a function on $M$. A solution of these equations is usually called a \textit{steady Euler flow}. As in the Euclidean case, we have the equivalent formulation
$$\uu \times \curl \uu =\nabla b\,, \qquad \Div \uu=0\,,$$
with $b:=\Pi+\frac12 |\uu|^2$, which implies, by \eqref{ident.cross}, that $\uu$ and $\cu\uu$ commute, i.e.,
\begin{equation}\label{commute}
[\uu, \cu \uu]=0.
\end{equation}
It is easy to check~\cite[Chapter II.1]{arn} that the Bernoulli function $b$ is a first integral of $\uu$ and $\cu \uu$ and that $\uu_x$ and $(\cu \uu)_x$ are linearly independent (nonvanishing) vectors
at each point $x \in M$ such that $(\nabla b)_x \neq 0$.

The following identities will be instrumental in the proof of Theorem~\ref{T.main1}.
\begin{lemma}
For any steady Euler flow $\uu$ on a Riemannian $3$-manifold $(M, g)$ we have the following identities
\begin{eqnarray}
&\abs{\uu}^2 \cu \uu  = g(\uu, \cu \uu)\uu - \uu \times \nabla b\,,  \label{ident0.1}\\
& \cu \uu (\abs{\uu}^2)  = \uu(g(\uu, \cu \uu))\,, \label{ident0.2}\\
& \cu \uu (g(\uu, \cu \uu))  = \uu(\abs{\cu \uu}^2)-g(\cu(\cu \uu), \nabla b)\,, \label{ident0.2bis}\\
& \uu(g(\uu, \cu \uu))\uu - \uu(\abs{\uu}^2) \cu \uu  = [\uu, \ \uu \times \nabla b]\,.  \label{ident0.3}
\end{eqnarray}
\end{lemma}
\begin{proof}
Take the cross product with $\uu$ in the stationary Euler equation $\uu \times \cu \uu =  \nabla b$ and use the identity~\eqref{triple} to get
\begin{equation*}
 \abs{\uu}^2 \cu \uu-g(\uu, \cu \uu)\uu =-\uu \times \nabla b,
\end{equation*}
which is~\eqref{ident0.1}. Then take the divergence in both members above and use the standard identity $\Div(fX)=X(f)+f\Div X$ and Equation~\eqref{ident.cross} to obtain
$$
(\cu \uu) (\abs{\uu}^2)  - \uu(g(\uu, \cu \uu)) =-g(\cu \uu , \nabla b)+g(\uu, \cu \nabla b),
$$
 which simplifies to Equation~\eqref{ident0.2} because $(\cu \uu) (b)=0$ and  $\cu \circ \nabla=0$. To prove the identity~\eqref{ident0.2bis} we follow the same steps but we start by the cross product with $\cu \uu$ in the stationary Euler equation.

Finally, to prove~\eqref{ident0.3} we take the Lie bracket with $\uu$ of Equation~\eqref{ident0.1} and use the identity~\eqref{commute}.
\end{proof}

\section{Proof of the main theorem}\label{sec:proof}

Let $\uu$ be a steady Euler flow as in the statement of Theorem~\ref{T.main1}. In particular, $\uu$ is an analytic vector field in a bounded domain $\Omega\subset\RR^3$, its Bernoulli function $b$ is non-constant, and $\uu$ is tangent to the boundary and localizable, i.e.,
$$\uu(\abs{\uu}^2)=0\,.$$
We observe that the function $b$ is also analytic, because it satisfies the elliptic PDE
\[
\Delta b = \Div(\uu\times \curl\uu)
\]
in $\Omega$, and the RHS is an analytic function. By assumption, each component of $\partial\Omega$ is a regular and compact level set of $b$. By continuity, $\nabla b\neq 0$ near $\partial\Omega$, so, in many parts of the proof, it will be very convenient to restrict ourselves to a neighborhood
$$\mathcal{W}\subset \Omega$$
of a compact connected component of a regular level set of $b$ (close to $\partial\Omega$) foliated by (component of) regular level sets. By construction, $(\nabla b)_x\neq 0$ and therefore $\uu_x\neq 0$ at each point $x\in\mathcal W$. Arnold's structure theorem~\cite{arn} implies that the (components of the) level sets of $b|_{\mathcal W}$ are diffeomorphic to tori, and the set $\mathcal W$ is diffeomorphic to $\mathbb T^2\times (0,1)$. In particular, since each component of $\partial\Omega$ is a regular level set of $b$, all the level sets near the boundary are diffeomorphic to each other, so we infer that each component of $\partial\Omega$ is diffeomorphic to a torus.

First, let us prove that the first integral $|\uu|^2$ is functionally dependent with $b$, and hence it is a function of $b$ in the regular neighborhood $\mathcal W$.

\begin{lemma}\label{BPfuncdep}
The first integrals of $\uu$, $|\uu|^2$ and $b$, are functionally dependent in $\Omega$. In particular, they are also first integrals of $\cu \uu$ in $\Omega$.
\end{lemma}

\begin{proof}
We argue by contradiction. Assume that the first integrals $b$ and $|\uu|^2$ are independent in the set $\mathcal W$. Since they are analytic functions, the set
$$\mathcal U:= \{x\in \mathcal W: \text{rank}\big((\nabla |\uu|^2)_x, \, (\nabla b)_x\big)=2\}$$
is open and dense in $\mathcal W$. Moreover, being first integrals, it is standard that $\mathcal U$ is an invariant set under the flow of $\uu$.
It follows that there exists a function $\mathcal{I}=\mathcal{I}(b, \abs{\uu}^2)$ such that
$$
\uu= \mathcal{I}(b, \abs{\uu}^2) \   \nabla |\uu|^2 \times \nabla b
$$
on $\mathcal U$. Since $\uu$ is nonvanishing on $\mathcal W$, we have that the function $\mathcal I$ does not vanish at any point in $\mathcal U$.

Taking the vector product of this expression with $\cu \uu$ we get:
$$\uu \times \cu \uu = \mathcal{I}(b, |\uu|^2) \, (\cu \uu)(|\uu|^2) \, \nabla b\,,$$
and since $\nabla b$ is nonvanishing on $\mathcal W$, the stationary Euler equation implies
$$\mathcal{I}(b, |\uu|^2) \, (\cu \uu)(|\uu|^2)=1\,,$$
which translates, using the identity~\eqref{ident0.2}, into
\begin{equation} \label{VS1}
\mathcal{I}(b, |\uu|^2) \, \uu(g(\uu, \cu \uu)) =1\,.
\end{equation}

Next, consider an integral curve $\phi_t(x)$ of the vector field $\uu$ in $\mathcal U$. This curve is periodic (say, of period $T$) by our assumption of two independent first integrals, and it is contained in $\mathcal U$. Setting $G:=g(\uu, \cu \uu)$, Equation~\eqref{VS1} reads along the flow $\phi_t(x)$ as
$$
\tfrac{d}{dt} G(\phi_t(x))=\frac{1}{\mathcal{I}(c_1,c_2)}\neq 0,
$$
where the constants $c_1$ and $c_2$ are the values of $b$ and $|\uu|^2$ along the integral curve $\phi_t(x)$.
By integrating this ODE, and taking into account the periodicity, we get
$$
G(x)= G(\phi_T(x)) = \frac{T}{\mathcal{I}(c_1,c_2)}+G(x),
$$
which is impossible, thus showing that the set $\mathcal U$ is empty, and hence the functions $b$ and $|\uu|^2$ are functionally dependent in $\mathcal W$. By analyticity, they are functionally dependent on the whole domain $\Omega$. Finally, since $|\uu|^2$ is a function of $b$ in $\mathcal W$ (because $\nabla b$ is nonvanishing), we conclude that $\cu \uu(|\uu|^2)=0$ on $\mathcal W$, and hence on the whole $\Omega$, again by analyticity.
\end{proof}
\begin{remark}\label{R:aw}
For later convenience, we notice here that $|\uu|^2$ is a function of $b$ in the set $\mathcal W$, i.e.,
\[
|\uu|^2=A_0(b)
\]
for some analytic function $A_0>0$. This is a straightforward consequence of the functional dependence of $b$ and $|\uu|^2$ and that $\nabla b$ is nonvanishing in $\mathcal W$.
\end{remark}



We divide the proof of Theorem~\ref{T.main1} in several steps. In Step~1 we introduce a convenient local chart of coordinates adapted to the vector field $\uu$ and the Bernoulli function $b$ and several functions and identities that will be key for the rest of the proof. In Step~2 we exploit that the Euclidean metric is flat, and obtain some geometric identities using that the curvature is zero. Step~3 is devoted to construct new first integrals and to prove that, in fact, they are functions of $b$. The previous results allow us to construct a symmetry vector field of $\uu$ in Steps~4 and~5, which is a Killing field of the metric. Finally, the previous local results are globalized in Step~6 using that all the involved functions are analytic. This proves that $\Omega$ is a rotationally symmetric domain and that $\uu$ is axisymmetric, so to complete the proof of Theorem~\ref{T.main1}, it remains to show the claim that the poloidal section of the rotationally symmetric domain $\Omega$ consists in a convex disk or a convex annulus, whose proof is presented in Section~\ref{S:polo}.

\subsection{Step~1: introduction of adapted local coordinates and some useful identities}\label{SS.step1}

The first observation is that if we combine Lemma~\ref{BPfuncdep} with Equation~\eqref{ident0.2}, we obtain that $g(\uu, \cu \uu)$ is also a first integral of $\uu$ on $\Omega$, i.e.,
$$\uu(g(\uu, \cu \uu))=0\,,$$
and, by Equation~\eqref{ident0.3}, we have the commutation relation
\begin{equation}\label{newcomm}
[\uu, \ \uu \times \nabla b]=0\,,
\end{equation}
in $\Omega$. Since $\uu$ and $\uu\times \nabla b$ commute, and are linearly independent in $\mathcal W$, we infer from~\cite[Theorem 9.46]{Lee} that around any point $x \in \mathcal{W}$ there are analytic (local) coordinates $(u, v, w)$ such that
\begin{equation}\label{coord_vect}
\uu=\partial_u, \qquad \uu \times \nabla b= \partial_v\,.
\end{equation}
In fact, the third coordinate $w$ is defined on the whole set $\mathcal W$ and indexes the nested toroidal level sets of $b|_{\mathcal W}$. Indeed, $\uu(b)=0$ and $(\uu \times \nabla b)(b)=0$ by Equation~\eqref{ident0.1}, so
\[
b=b(w)
\]
in $\mathcal W$. Analogously, since $\uu(|\uu|^2)=0$ by assumption, and $(\uu \times \nabla b)(|\uu|^2)=0$ by Lemma~\ref{BPfuncdep} and~\eqref{ident0.1}, we also have
$$|\uu|^2=A(w)$$
for some analytic function $A>0$ in $\mathcal W$. Comparing with Remark~\ref{R:aw}, we observe that
\[
A(w) = A_0(b(w))\,.
\]
Actually, the coordinate $w$ is not fixed and can be reparametrized as desired. The following choice of $w$ as a function of $b$ (the given Bernoulli function of $\uu$) turns out to be very convenient:
\[
w=w(b)=-\int_{0}^b A_0(s)ds\,,
\]
where $A_0$ is the function in Remark~\ref{R:aw}. It is then easy to check that $b(w)$ (the inverse of the function $w(b)$) and $A(w)$ are related as follows:
\begin{align}\label{eq.bA}
b'(w)=\frac{1}{w'(b(w))}=\frac{-1}{A_0(b(w))}=\frac{-1}{A(w)} \Longleftrightarrow b'(w)A(w)=-1\,.
\end{align}

\noindent Now let us remark that, since $g(\partial_u, \partial_v)=g(\uu,\uu \times \nabla b)=0$, in the (local) coordinates $(u,v,w)$  we can write the Euclidean metric tensor as
\begin{equation}\label{eq:det}
g=\left(\begin{array}{lll}
A & 0 & A U \\
0 & X & X V \\
A U & X V & g_{33}
\end{array}\right)\,.
\end{equation}
Here $X=|\uu \times \nabla b|^2>0$, $U=g_{13}/A$ and $V=g_{23}/X$ are analytic functions of $u,v,w$, and we have written the metric this way for later convenience, to obtain simpler formulas. In the following lemma we deduce from Equation~\eqref{eq.bA} that $\det(g)=1$ (i.e., the change of coordinates is volume-preserving), and therefore the function $g_{33}$ can be written in terms of the other functions:
$$g_{33}=\frac{1}{A X}+ A U^2 + X V^2\,.$$

\begin{lemma}
The determinant of the metric $g$ in coordinates $(u,v,w)$, cf. Equation~\eqref{eq:det}, is equal to $1$.
\end{lemma}
\begin{proof}
According to the standard choice of orientation in the Euclidean space, the frame $\{\uu, \nabla b, \uu \times \nabla b\}$ is positively oriented. As  $\uu(b)=0$, this frame is actually orthogonal, so evaluating the volume form on its normalized version gives
$$1=\mathrm{vol}_g\left(\tfrac{1}{\sqrt{A}}\uu, \tfrac{1}{\abs{\nabla b}}\nabla b,\tfrac{1}{\sqrt{A}\abs{\nabla b}} \uu \times \nabla b \right)\,,$$
or, equivalently,
$$
\mathrm{vol}_g\left(\uu, \uu \times \nabla b, \nabla b \right)=-A(w)\abs{\nabla b}^2\,.
$$
Next we write $\nabla b$ in the $(u,v,w)$ coordinates, i.e., $\nabla b=\vartheta_1 \partial_u+\vartheta_2 \partial_v + \vartheta_3 \partial_w$, for some analytic functions $\vartheta_1,\vartheta_2,\vartheta_3$. Here the only relevant function is $\vartheta_3$, which is easily computed by taking the scalar product of $\nabla b$ with itself, thus yielding
$$\vartheta_3 = |\nabla b|^2/b'(w)\,.$$
Then, using Equation~\eqref{coord_vect}, we obtain from the computation above
$$
\frac{|\nabla b|^2}{b'(w)}\mathrm{vol}_g(\partial_u, \partial_v, \partial_w)=- A(w) |\nabla b|^2\,,
$$
which, in view of Equation~\eqref{eq.bA} (our choice of the coordinate $w$), simplifies to
$$\mathrm{vol}_g(\partial_u, \partial_v, \partial_w)=-b'(w)A(w) =1\,.$$
Now let us show that $\{\partial_u, \partial_v,  \partial_w\}$ is positively oriented. Indeed the transition matrix from this frame to the positively oriented frame $\{\uu, \nabla b, \uu \times \nabla b\}$ has determinant $-\vartheta_3$, which is positive due to~\eqref{eq.bA}. Therefore, $\mathrm{vol}_g=\sqrt{\det(g)} \, du \wedge dv \wedge dw$ and, combining with $\mathrm{vol}_g(\partial_u, \partial_v, \partial_w)=1$ above, we obtain $\det(g)=1$.
\end{proof}

In what follows, we shall work with the local coordinates $(u,v,w)$. We recall they are defined in a neighbourhood of any point in $\mathcal W$.

\begin{lemma}
In the coordinates $(u,v,w)$, the vorticity (that is, $\cu \uu$) is given by
\begin{equation}\label{curlcoord}
\cu \uu = A \partial_v U \partial_u + \left(A^{\prime} - A \partial_u U\right)\partial_v\,.
\end{equation}
\end{lemma}

\begin{proof}
Using the metric $g$, we deduce
\begin{align*}
&(\partial_u)^\flat =A du +AU dw\,,\\
&(\partial_v)^\flat =X dv +XV dw\,,\\
&(\partial_w)^\flat =AUdu +XV dv +\left(\frac{1}{AX}+AU^2 + XV^2\right) dw\,.
\end{align*}
We can easily see that
$$
(\partial_u)^\flat \wedge (\partial_v)^\flat \wedge (\partial_w)^\flat = du\wedge dv \wedge dw= \mathrm{vol}_g \, .
$$
Using the Hodge star definition and the above relation, we obtain $\ast(\partial_u)^\flat(\partial_v, \partial_w)=1$ and $\ast(\partial_u)^\flat(\partial_w, \partial_u)=$ $\ast(\partial_u)^\flat(\partial_u, \partial_v)=0$, so $\ast(\partial_u)^\flat=dv \wedge dw$ or, equivalently, $(\partial_u)^\flat=\ast  (dv \wedge dw)$. Proceeding similarly for $\ast(\partial_v)^\flat$ and $\ast(\partial_w)^\flat$, we finally get:
\begin{equation}\label{wedges}
\left(\ast (dv \wedge dw)\right)^\sharp = \partial_u\,, \qquad
\left(\ast (dw \wedge du)\right)^\sharp = \partial_v\,, \qquad
\left(\ast (du \wedge dv)\right)^\sharp = \partial_w\,.
\end{equation}
Since $d\uu^\flat =d(\partial_u)^\flat =(A^\prime - A \partial_u U) dw \wedge du + A \partial_v U dv \wedge dw$, the above relations give us the expression~\eqref{curlcoord}.
\end{proof}

A straightforward computation gives
$$
\nabla w = AX\left(\partial_w-U\partial_u-V\partial_v\right)\,,
$$
and then
$$|\nabla w|^2 = \nabla w(w)=AX\,.$$
It is also convenient to introduce the following normal vector field:
\begin{equation}
\nn:=\frac{1}{|\nabla w|^{2}} \nabla w=\partial_w-U\partial_u-V\partial_v\,.
\end{equation}


In these coordinates the stationary Euler equation $\nabla_{\uu} \uu -\frac{1}{2}\nabla |\uu|^2 + \nabla b =0$ reduces to a single scalar equation:
\begin{equation}\label{dux}
\partial_u U=\frac{AA'+1}{A^2}=:\alpha(w)\,,
\end{equation}
so the (local) function $U$ integrates to
\begin{equation}\label{xdef+}
U(u, v, w)=\alpha(w) u  + \tilde \beta(v,w)\,,
\end{equation}
for some analytic function $\tilde\beta(v,w)$. In particular, we easily compute that
\begin{equation}\label{eq.GA}
G:=g(\uu, \cu \uu) = A(w)^2 \partial_v \tilde \beta (v,w)\,,
\end{equation}
which is a function that does not depend on the coordinate $u$, as we expect because we proved before that $\uu(g(\uu, \cu \uu))=0$. By computing $dG \wedge dw$ and using~\eqref{wedges}, we obtain:
\begin{equation}\label{grGgrw}
\nabla G \times \nabla w = (A^2 \partial_{vv} \tilde \beta) \ \uu\,.
\end{equation}
Let us remark that, although the function $\tilde \beta(v,w)$ is only locally defined, its derivatives $\partial_v \tilde \beta$ and $\partial_{vv} \tilde \beta$ are defined on the whole set $\mathcal W$.

Next, it is convenient to define several functions that will appear in several parts of the proof. Although they are introduced using the local coordinates $(u,v,w)$, actually we show in Lemma~\ref{coord_free_quanti} that they admit coordinate free expressions and are globally defined on the set $\mathcal W$.

\begin{equation}\label{not}
\begin{split}
& \bar{A}^{\prime}:= A^\prime -2\alpha A, \qquad \bar{A}^{\prime \prime}:=A''-\alpha  A' + 2 A \alpha ^2 - 2 (A \alpha)^{\prime}, \\
&  D :=\nn(\cdot), \qquad
\Omega:= DX - 2 X \partial_v V, \qquad z:=X \partial_u V,\\
& p:=\frac{\partial X}{\partial u}, \quad q:=\frac{\partial X}{\partial v}, \quad
r:=\frac{\partial p}{\partial u}, \quad s:=\frac{\partial p}{\partial v}=\frac{\partial q}{\partial u}, \quad t := \frac{\partial q}{\partial v}, \\
& F := \frac{p^2}{A X^2} + z^2 + K X, \quad K := \frac{(\bar{A}^{\prime})^2}{4 A}.
\end{split}
\end{equation}

The following lemma provides coordinate free expressions for several quantities that have been introduced before, thus showing that they are globally defined functions on the set $\mathcal W$:

\begin{lemma} \label{coord_free_quanti}
The following scalars are well defined on the whole domain $\mathcal{W}$:
\begin{align*}
& A=|\uu|^2, \quad \bar{A}^{\prime}=-2g(\nabla_{\uu} \uu, \nn), \quad X= \frac{|\nabla w|^2}{|\uu|^2}\,, \\
& \partial_u U =-\frac{1}{|\uu|^2}g\left([\uu, \nn], \uu\right)\,,
\quad \partial_v U =\frac{1}{|\uu|^4}g\left([\uu\times \nabla w, \nn], \uu\right)=\frac{1}{|\uu|^4}g(\uu, \cu \uu)\,,\\
& \partial_u V = \frac{1}{|\nabla w|^{2}} g\left([\uu, \nn], \uu\times \nabla w\right)\,, \\
& \partial_v V =-\frac{1}{|\uu|^2}\left(\nn(|\uu|^2)+\frac{1}{|\nabla w|^2}g\left([\uu\times \nabla w, \nn], \uu\times \nabla w\right) \right)\,.
\end{align*}
Consequently all quantities in~\eqref{not} are also globally defined on $\mathcal{W}$, as well as any derivatives of them with respect to $u$ or $v$. Moreover, the functions $\bar A'(w)$ and $K(w)$ do not vanish at any point in $\mathcal{W}$.
\end{lemma}

\begin{proof}
By definition, $A(w)$ is globally defined on $\mathcal W$, and the same with $X$, which is given by $X=|\nabla w|^2/A$. For $\bar{A}^{\prime}$ it is enough to notice that the pressure as a function of $w$ is given by
\[
\Pi(w)=b(w)-\frac12A(w)\,,
\]
and therefore the stationary Euler equation and Equation~\eqref{eq.bA} imply
\begin{equation}\label{eq.nablauu}
\nabla_{\uu} \uu =-\frac{\bar{A}^\prime}{2}  \nabla w\,,
\end{equation}
which yields the coordinate free expression of $\bar{A}^\prime$.

As to the derivatives of $U$ and $V$, they are like ``structure constants'' for the orthogonal frame $\{\uu, \uu\times \nabla w, \nn\}$. A straightforward computation using the previously defined objects shows us that we have:
\begin{equation}\label{commut_with_n}
\begin{split}
&[\uu, \nn]=-\alpha \, \uu + \frac{\partial_u V}{A}\, \uu\times \nabla w\,,\\
&[\uu\times \nabla w, \nn]=A \partial_v\tilde{\beta}\, \uu-\frac{A^\prime + A \partial_v V}{A}\, \uu\times \nabla w\,.
\end{split}
\end{equation}
Applying $g(\cdot, \uu)$ on these equations we immediately get the stated global expressions for $\alpha=\partial_u U$ and $\partial_v\tilde{\beta}=\partial_v U$. Similarly, applying $g(\cdot, \uu\times \nabla w)$ on these equations we obtain $\partial_u V$ and $\partial_v V$.

The claim that all the quantities in~\eqref{not} are globally defined follows if we notice that the derivatives of any global function with respect to the coordinates $u,v$ are derivatives along the global vector fields $\uu$ and $\frac{-1}{A}\uu \times \nabla w$, respectively, so they are globally defined on $\mathcal W$.

Finally, we show that $\bar A'$ (and hence $K$) is nonvanishing in $\mathcal{W}$. Indeed, let us assume the contrary, i.e., that $\bar A'$ is zero at $w=w_0$, which defines a toroidal surface invariant by the vector field $\uu$. The Equation~\eqref{eq.nablauu} becomes $\nabla_{\uu} \uu = 0$ on such a surface, which implies that any orbit of $\uu$ on the surface must be at the same time a geodesic line in $\RR^3$, so a straight line, which contradicts the compactness of the level set. Therefore, $\bar A'$ has no zeros in $\mathcal W$, as claimed.
\end{proof}

\subsection{Step~2: geometric identities using the flatness of the metric}

 Our goal in this section is to obtain several identities from the fact that the metric $g$ is flat (it is the Euclidean metric written in coordinates $(u,v,w)$). It is well known that in dimension $3$ the curvature is determined by the Ricci tensor, which has $6$ components. Equating them with zero gives us a system of $6$ necessary and sufficient conditions for flatness. Next, we present these equations, where we use the notations ``K'' and ``R'' for sectional curvature and Ricci, respectively. The reason for these names is that the corresponding equations are computed using that both the sectional curvature and the Ricci curvature are zero for different planes, although this particular distinction will not be relevant for us. Among the $6$ equations, we only write four of them, which are those that we will use in this step. Remember that the functions $A$ and $X$ are positive on $\mathcal W$.

\begin{lemma}\label{pal_ident_general}
The following identities hold true. The involved functions are analytic and globally defined on the set $\mathcal W$.
\begin{equation}\label{G3+}
\bar{A}'\Omega -(A \partial_v \tilde{\beta} + z)^2=\frac{p^2}{A X^2}-\frac{2 r}{A X}\,,
\tag{K3+}
\end{equation}
\begin{equation}\label{G2+}
\bar{A}' D X = \frac{2 r}{A X} - 3 z^2 - \frac{3 p^2}{A X^2} - 2 A z \partial_v \tilde{\beta}  + A^2 (\partial_v \tilde{\beta})^2 - 2\bar{A}'' X\,,
\tag{K2+}
\end{equation}
\begin{equation}\label{ICM3+}
\partial_u z+\frac{\bar{A}' q}{2 X}+\frac{p}{X}(A \partial_v \tilde{\beta} + z) = 0\,, \tag{R3+}
\end{equation}
\begin{equation}\label{IICM3+}
\partial_u \Omega + \partial_v z -\frac{\bar{A}' p}{2 A} +
A \partial_{vv} \tilde \beta=0\,.
 \tag{R2+}
\end{equation}
\end{lemma}

\begin{proof}
The equations~\eqref{IICM3+} and~\eqref{ICM3+} correspond, respectively, to zero Ricci curvature conditions, that is
$\Ric(\uu, \nabla w)=0$ and $\Ric(\uu \times \nabla w, \nabla w)=0$.
The equations~\eqref{G2+} and~\eqref{G3+} correspond, respectively, to zero sectional curvature conditions for the planes $\mathrm{span}\{\uu, \nabla w\}$ and $\mathrm{span}\{\uu, \uu \times \nabla w\}$.
The following curvature computations can be done using Wolfram Mathematica~\footnote{Wolfram Research, Inc., Mathematica, Version 14.3, Champaign, IL (2025)}. To this end, let us first compute the Levi-Civita connection. For the orthogonal frame
\begin{equation}\label{ortoframe}
\{E_1= \uu, \ E_2=\uu \times \nabla w, \ E_3=\nabla w\}\,,
\end{equation}
which is globally defined on $\mathcal W$, we can easily compute the covariant derivatives from the definition:
\begin{equation*}
\begin{split}
&\nabla_{E_1} E_1 =-\frac{\bar{A}^\prime}{2}  E_3\,, \quad
\nabla_{E_1} E_2 =\frac{p}{2 X} E_2 - \frac{A}{2} \big(A \partial_v \tilde{\beta} +z\big) E_3\,,\\
&\nabla_{E_1} E_3=  \frac{\bar{A}' X}{2}E_1 + \frac{1}{2}\big(A \partial_v \tilde{\beta} +z\big)E_2 + \frac{p}{2X} E_3\,,  \\
& \nabla_{E_2} E_2=-\frac{Ap}{2} E_1 -\frac{Aq}{2X} E_2 -\frac{A^2 \Omega}{2} E_3\,, \\
& \nabla_{E_2} E_3 =  \frac{AX}{2}\big(A \partial_v \tilde{\beta} +z\big)E_1 + \frac{A\Omega}{2}E_2 - \frac{Aq}{2X} E_3\,,\\
& \nabla_{E_3} E_3 = \frac{p}{2} E_1 -\frac{q}{2X} E_2 +\frac{1}{2}(A^\prime X+ A\, DX) E_3\,,
\end{split}
\end{equation*}
and the Lie brackets:
\begin{equation*}
\begin{split}
&[E_1, E_2]=0, \ [E_2, E_3]=A^2X \partial_v \tilde{\beta}  E_1 -X \left(A \partial_v V+A'\right)E_2 -\frac{Aq}{X}E_3\,, \\
&[E_3, E_1]= AX \alpha \,  E_1 -z\, E_2 -\frac{p}{X}E_3.
\end{split}
\end{equation*}
Observe that all the coefficients in these equations are globally defined on $\mathcal W$.

With these expressions in hand, the aforementioned curvature terms can be computed from the definition. We have
\begin{equation*}
\begin{split}
\Ric(\uu, \nabla w)&=\tfrac{1}{A^2 X}g(R(E_1, E_2)E_2, E_3)\\
&=\tfrac{1}{A^2 X}g\left(\nabla_{E_1} \nabla_{E_2} E_2 - \nabla_{E_2} \nabla_{E_1} E_2, \, E_3\right)\\
&= \tfrac{A}{2}\left(\partial_u \Omega + \partial_v z -\frac{\bar{A}' p}{2 A} +
A \partial_{vv} \tilde \beta\right)\,,
\end{split}
\end{equation*}
which leads us to Equation~\eqref{IICM3+}.
Similarly,
\begin{equation*}
\begin{split}
\Ric(\uu \times \nabla w, \nabla w)&=\tfrac{1}{A}g(R(E_2, E_1)E_1, E_3)\\
&=\tfrac{1}{A}g\left(\nabla_{E_2} \nabla_{E_1} E_1 - \nabla_{E_1} \nabla_{E_1} E_2, \, E_3\right)\\
&= \tfrac{A}{4}\left(2X\partial_u z+\bar{A}' q + 2p(A \partial_v \tilde{\beta} + z)\right),
\end{split}
\end{equation*}
which implies Equation~\eqref{ICM3+}. Denoting the curvature tensor by $R$, we also have,
\begin{equation*}
\begin{split}
g(R(\uu, \uu \times \nabla w)\uu \times \nabla w, \uu)&=g(R(E_1, E_2)E_2, E_1)\\
&=g\left(\nabla_{E_1} \nabla_{E_2} E_2 - \nabla_{E_2} \nabla_{E_1} E_2, \, E_1\right)\\
&= -\tfrac{A^3X}{4}\left(\bar{A}'\Omega -(A \partial_v \tilde{\beta} + z)^2-\frac{p^2}{A X^2}+\frac{2 r}{A X}\right)\,,
\end{split}
\end{equation*}
which yields Equation~\eqref{G3+}. Finally,
\begin{equation*}
\begin{split}
&g(R(\uu,  \nabla w) \nabla w, \uu)=g(R(E_1, E_3)E_3, E_1)\\
&=g\left(\nabla_{E_1} \nabla_{E_3} E_3 - \nabla_{E_3} \nabla_{E_1} E_3-\nabla_{[E_1, E_3]}E_3, \, E_1\right)\\
&= -\tfrac{A^2X}{4}\left(\bar{A}' D X - \frac{2 r}{A X} + 3 z^2 + \frac{3 p^2}{A X^2} + 2 A z \partial_v \tilde{\beta}  - A^2 (\partial_v \tilde{\beta})^2 + 2\bar{A}'' X\right)\,,
\end{split}
\end{equation*}
which leads us to Equation~\eqref{G2+}.
\end{proof}

In the next lemma, we use these zero curvature equations to prove an identity that will be key in Step~3 to show that certain first integrals are functions of the coordinate $w$.
\begin{lemma}
The following identity holds true:
\begin{equation}\label{dFdu+}
\partial_u F = \tfrac{1}{2}A\bar{A}^{\prime}\partial_{vv} \tilde \beta
\end{equation}
on the whole set $\mathcal W$.
\end{lemma}

\begin{proof}
For simplicity, we denote $\cB:=\partial_v \tilde{\beta}$. Recall that this is a function of $v$ and $w$. Adding and subtracting~\eqref{G2+} and~\eqref{G3+}, we obtain, respectively, these identities
\begin{align}
& \bar{A}^{\prime} \Omega = K X - F - \bar{A}^{\prime} X \partial_v V - \bar{A}^{\prime \prime} X +A^2 \cB^2 \label{G2plusG3+}\,,\\
& \bar{A}^{\prime} X \partial_v V = \frac{2 r}{A X}-2 F+\left(2K-\bar{A}^{\prime \prime}\right) X -2A\cB z\,.\label{G3minusG2+}
\end{align}
Now, take the derivative $\partial_u$ of~\eqref{G2plusG3+} and replace $\partial_u \Omega$ from Equation~\eqref{IICM3+}.  This way we get
\begin{align*}
&  \bar{A}^{\prime} \partial_u\Omega = K p - \partial_u F - \bar{A}^{\prime} \left(p \partial_v V + X \partial_{uv} V\right)-\bar{A}^{\prime \prime} p \Longleftrightarrow\\
& \bar{A}^{\prime} \left(-\partial_v z + \frac{\bar{A}' p}{2 A}-A \partial_v \cB \right) = K p - \partial_u F - \bar{A}^{\prime} \left(p \partial_v V + X \partial_{uv} V\right)-\bar{A}^{\prime \prime} p \Longleftrightarrow\\
& \bar{A}^{\prime} \left(-q\partial_u V -X\partial_{uv} V\right) +2Kp -A \bar{A}^{\prime}\partial_v \cB = K p - \partial_u F - \bar{A}^{\prime} \left(p \partial_v V + X \partial_{uv} V\right)-\bar{A}^{\prime \prime} p.
\end{align*}

By simplifying the second derivative of $V$ and rearranging the terms, we get
\begin{equation}\label{pqsimpleq}
\partial_u F+\left(K + \bar{A}^{\prime \prime}\right)p + \bar{A}^{\prime} p \partial_v V = \bar{A}^{\prime} q \partial_u V + A\bar{A}^{\prime}\partial_v \cB\,.
\end{equation}
Next, on the left hand side we replace $\bar{A}^{\prime} \partial_v V$ from Equation~\eqref{G3minusG2+}, and on the right hand side we replace $\bar{A}^{\prime}q$ from Equation~\eqref{ICM3+} and $\partial_u V$ by $\frac{z}{X}$, cf. the definition of the function $z$. One then obtains
\begin{align*}
& \partial_u F+\left(K + \bar{A}^{\prime \prime}\right)p + p \left(\frac{2 r}{A X^2}- \frac{2F}{X}+ 2K -\bar{A}^{\prime \prime} -2A\cB \frac{z}{X}\right)\\
&=\left(-2X\partial_u z-2p(A \cB + z)\right) \frac{z}{X} +A\bar{A}^{\prime}\partial_v \cB\,,
\end{align*}
or, equivalently
$$
\partial_u F+ p \left(\frac{2 r}{A X^2}- \frac{2F}{X}+ 3K \right)=-\frac{\partial z^2}{\partial u}-2z^2\frac{p}{X}+A\bar{A}^{\prime}\partial_v \cB\,.
$$
Finally, replace $z^2=F-\frac{p^2}{A X^2} - K X$ only in the derivative term $\frac{\partial z^2}{\partial u}$:
\begin{align*}
&  \partial_u F + p \left(\frac{2 r}{A X^2}- \frac{2F}{X}+ 3K \right)=-\partial_u\left(F-\tfrac{p^2}{A X^2} - K X\right)-2z^2\frac{p}{X} +A\bar{A}^{\prime}\partial_v \cB\Leftrightarrow\\
&  \partial_u F + p \left(\frac{2 r}{A X^2}- \frac{2F}{X}+ 3K \right)=-\partial_u F +2\frac{p}{A X} \cdot \frac{r X - p^2}{X^2} + K p-2z^2\frac{p}{X} +A\bar{A}^{\prime}\partial_v \cB\Leftrightarrow\\
&  \partial_u F + p \left(- \frac{2F}{X}+ 3K \right)=-\partial_u F -2\frac{p^3}{A X^3} + K p-2z^2\frac{p}{X} +A\bar{A}^{\prime}\partial_v \cB \Leftrightarrow \\
&  2\partial_u F + \frac{p}{X} \left(- 2F + 3KX + 2\frac{p^2}{A X^2}-KX +2z^2\right)=A\bar{A}^{\prime}\partial_v \cB\,.
\end{align*}
Since the parenthesis term vanishes by definition of the function $F$, this yields the desired Equation~\eqref{dFdu+}.
\end{proof}

\subsection{Step~3: a new first integral}

Our goal in this section is to show that the function $F$ is a first integral of $\uu$ that is functionally dependent with the Bernoulli function $b$. To this end, we first prove that $G:=g(\uu, \cu \uu)$, which is a first integral of $\uu$ (see the beginning of Section~\ref{SS.step1}), is also functionally dependent with $b$ (and hence a function of $b$, or equivalently of $w$, in the set $\mathcal W$).

\begin{lemma}\label{GPfuncdep}
The first integral $g(\uu, \cu \uu)$ is functionally dependent with $b$. In particular, it is also a first integral of $\cu \uu$. Moreover, $F$ is a first integral of $\uu$ wherever it is defined (specifically, in the set $\mathcal W$).
\end{lemma}

\begin{proof}
Since $G=g(\uu, \cu \uu) = A(w)^2 \partial_v \tilde \beta (v,w)$ by Equation~\eqref{eq.GA}, we can equivalently write Equation~\eqref{dFdu+} in coordinate free terms as
\begin{equation}\label{dFduINVAR}
\uu(F) =\vv(G)\,,  \qquad \vv := -\frac{\bar{A}^{\prime}}{2A^2} \uu \times \nabla w\,.
\end{equation}
Consider the flow $\phi_t$ of $\uu$, which is defined for all $t\in \mathbb R$ because $\mathcal W$ is invariant under the flow of $\uu$. The above relation rewrites
\begin{equation}\label{dFduflow}
\tfrac{d}{dt} F(\phi_t(x))=\vv(G)(\phi_t(x))\,.
\end{equation}
But
$$
\tfrac{d}{dt} \vv(G)(\phi_t(x))=\uu(\vv(G))\circ \phi_t(x)=\vv(\uu(G))\circ \phi_t(x)=0\,,
$$
where we have used that $[\uu, \vv]=0$ and that $G$ is a first integral of $\uu$. Therefore
$$\vv(G)(\phi_t(x))=C(x)$$
is independent of $t$, and from Equation~\eqref{dFduflow} we deduce that
$$F(\phi_t(x))=C(x)t + F(x)\,,$$
which contradicts the fact that $F$ is bounded, unless $C \equiv 0$, thus implying that $\uu (F)=0$ in $\mathcal W$ (and, by analyticity, wherever $F$ is defined). Finally, from Equation~\eqref{dFdu+} we obtain that $\partial_{vv} \tilde \beta=0$, and this fact together with Equation~\eqref{grGgrw} imply that $\nabla G \times \nabla w=0$ on $\mathcal W$. Accordingly, $G$ and $b$ (which is a function of $w$) are functionally dependent in $\mathcal W$.
\end{proof}

Since $\nabla w\neq 0$ at any point of $\mathcal W$, in view of Lemma~\ref{GPfuncdep} we can write
$$g(\cu \uu , \uu) = k(w)$$
for some analytic function $k$, which translates into
\begin{equation}\label{dvx}
\partial_v U = \frac{k(w)}{A^2}=:\beta(w)\,,
\end{equation}
(recall that, by definition, $\partial_v U = \partial_v \tilde \beta$). Therefore, from now on one can assume that locally,
\begin{equation}\label{xdef}
U(u, v, w)=\alpha(w) u  + \beta(w) v\,.
\end{equation}
This expression is valid only in the coordinate chart $(u,v,w)$, but the quantities $\partial_u U, \partial_v U$ are globally defined on the set $\mathcal{W}$, as they admit coordinate free expressions, cf. Lemma~\ref{coord_free_quanti}.

\begin{remark}\label{Rem_zerocurv}
Taking into account the local equation~\eqref{xdef}, the zero curvature conditions in Lemma~\ref{pal_ident_general} get a simpler form. We supplement them with the remaining two that we did not include in the aforementioned lemma. The new identities~\eqref{G1} and~\eqref{IICM1} correspond, respectively, to zero sectional curvature condition for the plane $\mathrm{span}\{\uu \times \nabla w, \nabla w\}$, and to zero Ricci curvature condition $\Ric(\uu, \uu \times \nabla w)=0$.
\begin{equation}\label{G3}
\bar{A}'\Omega -(A \beta + z)^2=\frac{p^2}{A X^2}-\frac{2 r}{A X}\,,
\tag{K3}
\end{equation}

\begin{equation}\label{G2}
\bar{A}' D X = \frac{2 r}{A X} - 3 z^2 - \frac{3 p^2}{A X^2} - 2 A \beta z + A^2 \beta^2 - 2\bar{A}^{\prime \prime} X\,,
\tag{K2}
\end{equation}

\begin{equation}\label{G1}
 -\frac{p^2}{AX^2}-z^2+2 A \beta z = \left(2 A \partial_v V - A^{\prime}\right) \Omega
-2 A \, D \Omega + \frac{2}{X^2} \partial_v q - \frac{4 q^2}{X^{3}}-3 A^{2} \beta^2\,,
\tag{K1}
\end{equation}

\begin{equation}\label{ICM3}
\partial_u z+\frac{\bar{A}' q}{2 X}+\frac{p}{X}(A \beta + z) = 0\,,
\tag{R3}
\end{equation}

\begin{equation}\label{IICM3}
\partial_u \Omega + \partial_v z -\frac{\bar{A}' p}{2 A}=0\,,
\tag{R2}
\end{equation}

\begin{equation}\label{IICM1}
\frac{1}{A X^2} \partial_v p - \frac{2 p q}{A X^3}+\frac{z}{X} \Omega+D z+(A \beta)^{\prime}-2 A \alpha \beta+ A' \beta =0\,.
\tag{R1}
\end{equation}
It is important to emphasize that, although we used the local equation~\eqref{xdef} to obtain these identities, all the quantities involved in these equations are globally defined (see Lemma~\ref{coord_free_quanti}), so they hold on the whole set $\mathcal{W}$.
\end{remark}

By Lemma~\ref{GPfuncdep} we know that $F$ is a first integral of $\uu$, i.e.,
\begin{equation}\label{dFdu}
\partial_u F = 0\,.
\end{equation}
We claim that, in fact, $F=F(w)$ on $\mathcal W$. To prove this, it is convenient to work in the open set
$$\mathcal{W}_{(p)} := \{x\in \mathcal{W}:\, p(x)\neq 0\}\,.$$
Actually, $\mathcal W_{(p)}$ is dense in $\mathcal W$. Indeed, by analyticity, it is enough to show that $\mathcal{W}_{(p)} \neq \emptyset$. Assume that
$p=0$ on a whole toroidal surface $\{w=w_0\}$. The following lemma shows that the Gaussian curvature of the level set $\{w=w_0\}$ vanishes:

\begin{lemma}
Assume that $p=0$ on the surface $\{w=w_0\}$. Then the Gaussian curvature $\mathcal K=0$ on the toroidal surface $\{w=w_0\}$.
\end{lemma}
\begin{proof}
Since $p=\partial_u X$, then $\partial_{uu} X=0$ on $\{w=w_0\}$, and hence the Gaussian curvature of the
surface, which is given by the formula
\begin{equation}\label{gauss_int}
\mathcal{K} = \frac{(\partial_u X)^2-2 X \partial_{uu} X}{4 A X^2}\,,
\end{equation}
is $0$.
To obtain Equation~\eqref{gauss_int}, we first write the induced metric on the surface $\{w=w_0\}$ from the ambient metric $g$, which is
$$g|_{w=w_0}=A(w_0)du^2+X(u,v,w_0)dv^2\,.$$
Then, it is enough to recall that, for surfaces, the Gaussian curvature is half of the scalar curvature of the metric (Gauss' Egregium Theorem), which can be computed in a straightforward way, leading to the claimed equation.
\end{proof}

It is well known that an immersed compact surface in $\mathbb R^3$ cannot have zero Gaussian curvature, so we infer that $p$ cannot vanish on the whole level set $\{w=w_0\}$, as we wanted to show.

Before proving the main result of this step, we have to establish several identities that are crucial in the proof. This is the content of the following lemma. We observe that all the functions in the statement are well defined on the whole set $\mathcal{W}_{(p)}$. In particular, we observe that $\bar{A}^{\prime}$ is nonvanishing on $\mathcal W$ (cf. Lemma~\ref{coord_free_quanti}) and $F-KX>0$ on $\mathcal{W}_{(p)}$.

\begin{lemma}\label{qoverp}
In the open and dense set $\mathcal{W}_{(p)}$, the following identities hold true:
\begin{equation}\label{derivz}
\partial_v z =\frac{q}{p} \partial_u z +
z \partial_u \!\! \left(\frac{q}{p} \right).
\end{equation}

\begin{equation}\label{Fqp}
\partial_v F = 2 \left(F-KX\right)\partial_u \!\! \left(\frac{q}{p} \right).
\end{equation}

\begin{equation}\label{dvqp}
\partial_v\!\left(\frac{q}{p}\right)
= 2\left(\frac{q}{p}
+\frac{A\beta}{\bar{A}^{\prime}}
-\frac{KXz}{\bar{A}^{\prime}(F-KX)}\right)\partial_{u}\!\left(\frac{q}{p}\right).
\end{equation}
\end{lemma}

\begin{proof}
By definition of $F$,
$$\partial_u F = \frac{2p}{AX} \cdot \frac{X \partial_u p -p^2}{X^2}+2 z\partial_u z +Kp$$
and
$$\partial_v F= \frac{2p}{AX} \cdot \frac{X \partial_v p -pq}{X^2}+2 z\partial_v z +Kq\,,$$
so using that $F$ is a first integral of $\uu$, i.e., $\partial_u F=0$, the linear combination $p\partial_v F-q\partial_u F$ reads:
\begin{equation}\label{sillyF}
p\partial_v F = \frac{2p}{AX^2}\left(p \partial_u q - q\partial_u p \right) +2z\left(p\partial_v z - q\partial_u z\right)\,.
\end{equation}
On the set $\mathcal{W}_{(p)}$ we can divide by $p$ and rewrite the above identity as:
\begin{equation}\label{sillyFoverp}
\partial_v F = 2 \frac{p^2}{AX^2}\partial_u \!\! \left(\frac{q}{p} \right) +2z\left(\partial_v z - \frac{q}{p} \partial_u z\right)\,.
\end{equation}

Since $\partial_u F=0$ and $\partial_{vv} \tilde{\beta}=0$, Equation~\eqref{pqsimpleq} reads
\begin{equation}\label{pqsimpleq_reduced}
\left(K + \bar{A}^{\prime \prime} + \bar{A}^{\prime} \partial_v V\right)p = (\bar{A}^{\prime} \partial_u V)\, q\,,
\end{equation}
and replacing $\bar{A}^{\prime}  \partial_v V$ from Equation~\eqref{G2plusG3+} and $\partial_u V$ by $z/X$, we obtain
\begin{equation}\label{pq_proportional}
-\bar{A}^\prime qz =(F -2KX + \bar{A}^{\prime}\Omega - A^2 \beta^2)p\,.
\end{equation}
So, in $\mathcal{W}_{(p)}$,
$$
-\bar{A}^\prime z \frac{q}{p}
= F - 2KX + \bar{A}^{\prime}\Omega - A^2 \beta^2.
$$
Then, taking the derivative $\partial_u$ of this identity, using again that $\partial_u F =0$ and replacing $\partial_u \Omega$ from Equation~\eqref{IICM3}, we conclude that
\begin{align*}
& -\bar{A}^\prime \frac{q}{p} \partial_u z -\bar{A}^\prime z\partial_u\left(\frac{q}{p}\right)
=  - 2Kp + \bar{A}^{\prime}\left(-\partial_v z + \frac{\bar{A}' p}{2 A}\right)\,,
\end{align*}
which after rearrangement of terms and simplifications gives Equation~\eqref{derivz}. To obtain Equation~\eqref{Fqp}, we simply take $\partial_v z - \frac{q}{p} \partial_u z$ from Equation~\eqref{derivz} and inject it into Equation~\eqref{sillyFoverp}.

It remains to prove Equation~\eqref{dvqp}. Multiplying Equation~\eqref{ICM3} by $2X/(p\bar{A}^{\prime})$, we obtain the following identity in $\mathcal{W}_{(p)}$:
\begin{equation}\label{fracqp}
\frac{q}{p}+\frac{2}{\bar{A}^{\prime}} \left( \frac{X \partial_u z}{p}+A \beta+z\right)=0\,.
\end{equation}
Additionally, from Equation~\eqref{derivz} we get
\begin{equation}\label{derivz+}
\partial_v z = \partial_u \!\! \left(\frac{q}{p}\, z \right)\,,
\quad
\partial_{uv} z = \partial_{uu} \!\! \left(\frac{q}{p}\, z \right)\,.
\end{equation}

\noindent Applying $\partial_v$ on Equation~\eqref{fracqp}, replacing the derivatives of $z$ using the previous equations~\eqref{derivz+} and using derivatives comutation $\partial_v p=\partial_u q$, we successively obtain
\[
\begin{aligned}
&\partial_v\!\left(\frac{q}{p}\right)
+\frac{2}{\bar{A}^{\prime}}\left[
\frac{\bigl(q\,\partial_u z+X\,\partial_{u v} z\bigr)p
- X\,\partial_u z \cdot \partial_u q}{p^2}
+\partial_v z\right]=0\,, \\[0.4em]
& \;
\partial_v\!\left(\frac{q}{p}\right)
+\frac{2}{\bar{A}^{\prime}}
\left[
\frac{q}{p}\,\partial_u z+\frac{X}{p}\,\partial_{uv} z
- \frac{X}{p^2}\,\partial_u z\,\partial_u q +\partial_v z
\right]=0\,, \\[0.4em]
& \;
\partial_v\!\left(\frac{q}{p}\right)
+\frac{2}{\bar{A}^{\prime}}
\left[
\frac{q}{p}\,\partial_u z
+\frac{X}{p}\left(
\frac{q}{p}\,\partial_{uu}z
+2\,\partial_u z\,\partial_u\!\left(\frac{q}{p}\right)
+z\,\partial_{uu}\!\left(\frac{q}{p}\right)
\right)\right.\\[0.4em]
& \qquad \qquad \qquad \left.
-\frac{X\,\partial_u z\,\partial_u q}{p^2}
+\frac{q}{p}\,\partial_u z
+z\,\partial_u\!\left(\frac{q}{p}\right)
\right]=0\,,
\end{aligned}
\]
\[
\begin{aligned}
&\partial_v\left(\frac{q}{p}\right)
+\frac{2}{A^{\prime}}\left[
2 \frac{q}{p} \partial_u z
+\frac{Xq^2}{p^2} \cdot \frac{q  \partial_{u u} z-\partial_u z  \partial_u q}{q^2}
+\frac{2 X}{p} \partial_{u} z \,  \partial_u \!\!\left(\frac{q}{p}\right)
+\frac{X z}{p} \partial_{u u}\left(\frac{q}{p}\right)
+z \partial_u\left(\frac{q}{p}\right)
\right]=0\,, \\[0.4em]
& \;
\partial_v\left(\frac{q}{p}\right)
+\frac{2}{A^{\prime}} \left[
2 \frac{q}{p} \partial_u z
+\frac{X q^2}{p^2} \partial_u\!\!\left(\frac{\partial_u z}{q}\right)
+\frac{2 X}{p} \partial_{u} z \partial_u\left(\frac{q}{p}\right)
+\frac{X z}{p} \partial_{u u}\left(\frac{q}{p}\right)
+z \partial_u\left(\frac{q}{p}\right)
\right]=0\,,
\end{aligned}
\]
$$
\partial_v\left(\frac{q}{p}\right)
+\frac{2}{A^{\prime}}\left[
2 \frac{q}{p} \partial_{u} z
+\partial_u\left(\frac{X q \partial_{u} z}{p^2}\right)
-\frac{\partial_u z}{q} \left(
p \frac{q^2}{p^2}+2 X \frac{q}{p} \partial_u\left(\frac{q}{p}\right)
\right)\right.
$$
\[
\begin{aligned}
& \qquad \qquad \qquad \left. +\frac{2 X}{p} \partial_u z \partial_u\left(\frac{q}{p}\right)
+\frac{Xz}{p} \partial_{u u}\left(\frac{q}{p}\right)
+z \partial_u\left(\frac{q}{p}\right)
\right]=0\,, \\[0.4em]
& \;
\partial_v\!\left(\frac{q}{p}\right)
+\frac{2}{\bar{A}'}\left[
\frac{q}{p}\,\partial_u z
+\partial_u\!\left(\frac{X q\,\partial_u z}{p^2}\right)
-\frac{2X}{p}\,\partial_u z\,\partial_u\!\left(\frac{q}{p}\right)\right.\\[0.4em]
& \qquad \qquad \qquad \left.
+\frac{2X}{p}\,\partial_u z\,\partial_u\!\left(\frac{q}{p}\right)
+\frac{X z}{p}\,\partial_{uu}\!\left(\frac{q}{p}\right)
+z\,\partial_u\!\left(\frac{q}{p}\right)
\right]=0\,, \\[0.4em]
& \partial_v\!\left(\frac{q}{p}\right)
+\frac{2}{\bar{A}'}\left[
\partial_u\!\left(
\frac{q}{p}z+\frac{Xq\,\partial_u z}{p^2}
\right)
+\frac{Xz}{p}\,\partial_{uu}\!\left(\frac{q}{p}\right)
\right]=0\,.
\end{aligned}
\]
Next, in this equation we replace $\partial_u z$ using Equation~\eqref{ICM3}, which yields
$$
\begin{aligned}
\Rightarrow\;&
\partial_v\!\left(\frac{q}{p}\right)
+\frac{2}{\bar{A}'}\Bigg[
\partial_u\!\left(
\frac{q}{p}z
+\frac{Xq}{p^2}
\left(
-\frac{\bar{A}'q}{2X}
-\frac{p(A\beta+z)}{X}
\right)
\right)
+\frac{Xz}{p}\,\partial_{uu}\!\left(\frac{q}{p}\right)
\Bigg]
=0\,,
\\[6pt]
\Rightarrow\;&
\partial_v\!\left(\frac{q}{p}\right)
+\frac{2}{\bar{A}'}\Bigg[
\partial_u\!\left(
-\frac{\bar{A}'}{2}\frac{q^2}{p^2}
-A\beta\frac{q}{p}
\right)
+\frac{Xz}{p}\,\partial_{uu}\!\left(\frac{q}{p}\right)
\Bigg]
=0\,.
\end{aligned}
$$
Finally, applying $\partial_u$ on Equation~\eqref{Fqp} and using that $\partial_u F=0$, we get
\[
0=\frac{\partial^2 F}{\partial u\,\partial v}
=2\left[-Kp\,\partial_u\!\left(\frac{q}{p}\right)
+(F-KX)\,\partial_{uu}\!\left(\frac{q}{p}\right)\right]\,,
\]
and thus on the set $\mathcal{W}_{(p)}$ we have
\[
\partial_{uu}\!\left(\frac{q}{p}\right)
=\frac{Kp}{F-KX}\partial_u\!\left(\frac{q}{p}\right)\,.
\]
Injecting this equation into the last form of our identity above we obtain
$$
\partial_v\!\left(\frac{q}{p}\right)
= \partial_u \left(\frac{q^2}{p^2}
+2\frac{A\beta}{\bar{A}^{\prime}}
\frac{q}{p}\right)-2\frac{KXz}{\bar{A}^{\prime}(F-KX)}\partial_{u}\!\left(\frac{q}{p}\right)\,,
$$
which immediately gives Equation~\eqref{dvqp}.
\end{proof}

We are ready to prove the main result of this section. Recall that the function $F$ is analytic and globally defined
(at least) on the set $\mathcal W$.

\begin{lemma}\label{BFfuncdep}
The first integral $F$ is functionally dependent with $b$. In particular, it is also a first integral of $\cu \uu$.
\end{lemma}

\begin{proof}
We argue as in the proof of Lemma~\ref{BPfuncdep}. Assume that the first integrals $b$ and $F$ are independent in the set $\mathcal W$. Since they are analytic functions, the set
$$\mathcal U:= \{x\in \mathcal W: \text{rank}\big((\nabla F)_x,(\nabla b)_x\big)=2\}$$
is open and dense in $\mathcal W$. Moreover, being first integrals, the set $\mathcal U$ is invariant under the flow of $\uu$.
Then, there exists a function $\mathcal{I}=\mathcal{I}(b, F)$ such that
$$
\uu= \mathcal{I}(b, F) \   \nabla F \times \nabla b
$$
on $\mathcal U$. Since $\uu$ is nonvanishing on $\mathcal W$, we have that the function $\mathcal I$ does not vanish at any point in $\mathcal U$. Moreover, the set $\mathcal U$ is fibred by period orbits of $\uu$.

Taking the vector product with $\cu \uu$ we obtain
$$\uu\times \cu  = \mathcal{I}(b, F) (\cu \uu)(F) \, \nabla b\,,$$
and therefore
$$\mathcal{I}(b, F) \, (\cu \uu)(F)=1$$
on $\mathcal U$.
Since Equation~\eqref{ident0.1} implies that
\[
(\cu \uu)(F)= -\frac{1}{\abs{\uu}^2} (\uu\times \gr b)(F)
= b'(w) \partial_v F\,,
\]
we obtain
\begin{equation} \label{VS}
\mathcal{I}(b, F)b'(w) \partial_v F =1\,.
\end{equation}
On the other hand, let us observe that from Equation~\eqref{ICM3} we have
$$
\begin{aligned}
& X \partial_u z + \frac{\bar{A}^{\prime}}{2} q + p z + p A \beta=0 \Rightarrow q=-\frac{2}{\bar{A}^{\prime}}\left(X \partial_u z + z\partial_u X+p A \beta\right) \Rightarrow\\
& q =-\frac{2}{\bar{A}^{\prime}}\left(\partial_u(X z)+\partial_u(A \beta X)\right) =-\partial_u\!\!\left(\frac{2}{\bar{A}^{\prime}}(A \beta+z) X\right)\,.
\end{aligned}
$$
Combining this observation with Equation~\eqref{Fqp} we infer that
\begin{equation}\label{eq.use}
\partial_v F =2 K q+ 2 \partial_u \left((F-KX) \frac{q}{p}\right)
=2 \frac{\partial}{\partial u}\left(-\frac{2 K}{\bar{A}^{\prime}}(A \beta+z) X+\widetilde{\gamma}\right),
\end{equation}
which holds on an open and dense subset of $\mathcal U$. Here,
$$\widetilde{\gamma}:=(F-KX) \frac{q}{p}= \frac{pq}{AX^2}+\frac{qz^2}{p}$$
is a function, which is in fact globally defined on $\mathcal W$ due to the identity~\eqref{pq_proportional}. Accordingly, by continuity, Equation~\eqref{eq.use} is well defined on the whole set $\mathcal U$.

To complete the proof, we consider an integral curve $\phi_t(x)$ of the vector field $\uu$ in $\mathcal U$, which is periodic. Along such an orbit, Equation~\eqref{VS}, together with~\eqref{eq.use}, yield
$$
\frac{d}{d t}\left[-\frac{2 K}{\bar{A}^{\prime}}(A \beta+z(\phi_t(x))) X(\phi_t(x))+\widetilde{\gamma}(\phi_t(x))\right]=\frac{1}{C},
$$
where $C=2I(b, F) b^{\prime}(w)\neq 0$ is a constant. This is of course impossible on a periodic orbit, and therefore the set $\mathcal U$ is empty, thus showing that $F$ and $b$ are functionally dependent on $\mathcal W$ (and everywhere where $F$ is defined by analyticity). Proceeding as in the proof of Lemma~\ref{BPfuncdep}, we also conclude that $F$ is a first integral of $\cu\uu$.
\end{proof}

\subsection{Step~4: construction of a Killing symmetry}
Lemma~\ref{BFfuncdep} implies that $F$ is a function of $b$, and hence of $w$, in the set $\mathcal W$. Combining this fact with Equations~\eqref{Fqp} and~\eqref{dvqp}, we obtain that
\begin{equation}
\partial_u (q/p) = \partial_v(q/p)=0
\end{equation}
in an open and dense subset of $\mathcal W$. Since $p$ and $q$ are analytic functions, it follows that
$$q/p=\gamma^*(w)$$
for some function $\gamma^*$. In fact, $\gamma^*$ is analytic in $\mathcal{W}$. Indeed, if $\gamma^*$ were not defined for some value $w=w_0$, it would imply that $p=0$ on the whole toroidal surface $\{w=w_0\}$, which is impossible, as we proved before Lemma~\ref{qoverp}.

Next, we have to consider two different cases:
\begin{enumerate}
\item The function $q$ is zero, $q\equiv 0$, on the whole set $\mathcal W$, or
\item $q$ is not identically zero.
\end{enumerate}
The first case will be studied in Step~5, so in what follows let us assume that $q$ is not zero everywhere. Then, the function $\gamma^*(w)$ is not identically zero in $\mathcal W$, and being analytic, its zero set consists of finitely many values of $w$. Therefore, changing the set $\mathcal W$ if necessary, there is no loss of generality in assuming that $\gamma^*$ is nonvanishing on $\mathcal W$.

Since $\gamma^*\neq 0$ on $\mathcal W$, we can take the inverse to obtain
$$\frac{p}{q}=\frac{1}{\gamma^*(w)}=:\gamma(w)\,,$$
which is equivalent, taking into account the definition of the functions $p$ and $q$, to the following first order linear PDE for the function $X$:
$$\partial_u X -\gamma \partial_v X=0\,.$$
A straightforward application of the method of characteristics leads to
\begin{equation}\label{Xgam}
X= X(\gamma(w)u+v, w)\,,
\end{equation}
which holds on any local chart $(u,v,w)$ covering the set $\mathcal W$.

In the following lemma we prove that the function $\gamma(w)$ satisfies a first order ODE (of Bernoulli type). We recall that the functions $\alpha\equiv \alpha(w)$, $K\equiv K(w)$, $\bar{A}^{\prime}\equiv \bar{A}^{\prime}(w)$, $\bar{A}^{\prime \prime}\equiv \bar{A}^{\prime \prime}(w)$ were introduced in Section~\ref{SS.step1} and $\beta\equiv \beta(w)$ in Equation~\eqref{dvx}.

\begin{lemma}\label{RiccatiLemma}
The analytic function $\gamma\equiv \gamma(w)$ satisfies the following ODE
 \begin{equation}\label{Riccati}
\gamma^\prime - (\alpha-\beta\gamma)\gamma=\gamma\,\frac{K + \bar{A}^{\prime \prime}}{\bar{A}^{\prime}}
\end{equation}
in $\mathcal W$. Additionally, the following identity holds true
\begin{equation}\label{compat}
\gamma^\prime  -(\alpha-\beta \gamma)\gamma + \gamma \partial_v V - \partial_u V=0.
\end{equation}
\end{lemma}

\begin{proof}
Since $p = \gamma(w)q$ and $F=F(w)$ in the set $\mathcal W$, it follows that Equation~\eqref{pqsimpleq_reduced} is equivalent to
\begin{equation}\label{PDEy}
\left(K + \bar{A}^{\prime \prime}\right)\gamma + \bar{A}^{\prime} \gamma \partial_v V = \bar{A}^{\prime}  \partial_u V\,,
\end{equation}
and taking the derivative $\partial_u$, it implies
\begin{equation}\label{PDE2y}
\gamma \partial_{uv} V =  \partial_{uu} V\,.
\end{equation}

Next, we introduce the following function, which is identically zero according to Equation~\eqref{G2},
$$
\mathcal{G}_2:=\bar{A}' D X - \frac{2 r}{A X} + 3 z^2 + \frac{3 p^2}{A X^2} + 2 A \beta z - A^2 \beta^2 + 2\bar{A}^{\prime \prime} X\equiv 0\,,
$$
and the vector field
$$\xi_0 := \partial_u - \gamma\partial_v$$
on $\mathcal W$. Observe that Equations~\eqref{derivz} and~\eqref{Xgam} imply
$$
\xi_0(z)=\xi_0(X)=\xi_0(p)=\xi_0(r)=0\,.
$$
Now recall that $DX=\partial_w X - Up-Vq$ by definition of the involved quantities, and notice that $\partial_u X -\gamma \partial_v X=0$ implies
$$\partial_{uw} X -\gamma \partial_{vw} X=\gamma^\prime q$$
as well as
$$\partial_{uu} X -\gamma \partial_{uv} X=0\,, \qquad \partial_{uv} X -\gamma \partial_{vv} X=0\,.$$
Therefore, a straightforward computation shows that
\begin{equation*}
\begin{split}
\xi_0(DX)&=q\gamma' -(\alpha -\beta \gamma) p - q(\partial_u V-\gamma \partial_v V)\,.
\end{split}
\end{equation*}
These identities allow us to check that $\partial_u \mathcal{G}_2-\gamma \partial_v \mathcal{G}_2 \equiv 0$ is equivalent to the equation
$$
\bar{A}^\prime q  \, \left(\gamma' - \left(\alpha - \beta \gamma\right)\gamma  -\partial_u V  + \gamma\partial_v V\right)=0\,,
$$
thus proving~\eqref{compat} (because $\bar{A}^\prime$ and $q$ do not vanish in $\mathcal W$). Finally, combining with Equation~\eqref{PDEy} we obtain~\eqref{Riccati}, which is a Bernoulli ODE giving us $\gamma$ in terms of metric coefficients.
\end{proof}

We are ready to construct a Killing symmetry of the vector field $\uu$ in the set $\mathcal W$.

\begin{lemma}\label{killingness}
The vector field
$$\xi := \gamma_1(w)\left(\partial_u - \gamma(w)\partial_v\right)$$
is a Killing field (i.e., $\mathcal{L}_\xi g=0$) commuting with $\uu$ in $\mathcal W$, where
$$\gamma_1(w) := \exp \left(\int _1^w(\beta (t) \gamma (t)-\alpha (t))dt\right)\,.$$
\end{lemma}

\begin{proof}
In order to compute the Lie derivative $\mathcal{L}_\xi g$, we use the orthogonal frame $\{E_1,E_2,E_3\}$ introduced in Equation~\eqref{ortoframe}. The components of $\mathcal{L}_\xi g$ in this frame can be obtained using the definition:
$$\left(\mathcal{L}_\xi g\right)(E_i,E_j)=\xi(g(E_i,E_j)) - g([\xi,E_i],E_j) - g(E_i, [\xi, E_j])\,.$$
We first notice that
$$|E_1|^2 =A(w)\,,\qquad |E_2|^2 =A(w)^2 X(\gamma (w)u + v,w)\,,\qquad |E_3|^2 =A(w) X(\gamma (w)u + v,w)\,,$$
and they are all constant along the flow lines of $\xi$:
\begin{equation}\label{constNORMS}
\xi(|E_i|^2)=0\,, \quad i=1,2,3.
\end{equation}
Moreover, it is elementary to check that
\begin{equation*}\label{commXI}
\begin{split}
&[\xi, E_1]=[\xi, E_2]=0\,, \\
&[\xi, E_3] = -A X \left((\alpha -\beta \gamma)\gamma_1 +\gamma_1^\prime\right)E_1
- X\left(\gamma_1^\prime\gamma +\gamma_1(\gamma^\prime + \gamma \partial_v V - \partial_u V) \right)E_2\,.
\end{split}
\end{equation*}
Using the fact that the function $\gamma_1(w)$ satisfies, by definition, the ODE
\begin{equation}\label{gamma1eq}
(\alpha-\beta \gamma)\gamma_1+\gamma_1^\prime=0\,,
\end{equation}
and combining this with Equation~\eqref{compat}, we see that actually $\xi$ and $E_3$ commute, that is
$$
[\xi, E_3] =0\,.
$$
It clearly follows from the fact that $\xi(\abs{E_i}^2)=0$ and $[\xi, E_i] =0$ for any $i=1,2,3$, that $(\mathcal{L}_\xi g)(E_i, E_j)=0$ for all $i, j=1,2,3$, thus implying that $\xi$ is a Killing vector field, as we wanted to prove.
\end{proof}

\subsection{Step~5: Construction of a Killing symmetry in a degenerate case}\label{SS.step5}

In this section we consider the exceptional case $q\equiv 0$ in $\mathcal W$. Our goal is to prove a result analogous to Lemma~\ref{killingness}, i.e., the existence of a Killing symmetry of $\uu$. To this end, the following result will be crucial. Recall that $\beta\equiv \beta(w)$ was introduced in Equation~\eqref{dvx}.

\begin{lemma}\label{L.q0}
If $q=0$ in $\mathcal W$, then $\beta=0$ in $\mathcal W$.
\end{lemma}
\begin{proof}
By Equation~\eqref{pqsimpleq_reduced}, as $p\neq 0$ in an open and dense set, we must have that $K + \bar{A}^{\prime \prime} + \bar{A}^{\prime} \partial_v V=0$ so
\begin{equation}\label{partialvy0}
\partial_v V=-\frac{K + \bar{A}^{\prime \prime} }{\bar{A}^{\prime}}\,.
\end{equation}
(Recall that $\bar A^\prime$ is nonvanishing). In particular, since this relation implies that $\partial_v V$ is a function of $w$ only, we obtain
$$\partial_v z= \partial_v(X \partial_u V)=q\partial_u V + X \partial_{uv}V=0\,.$$
Now, from Equation~\eqref{pq_proportional} we also get
$$F -2KX + \bar{A}^{\prime}\Omega - A^2 \beta^2=0\,,$$
whose derivative $\partial_u, \partial_v$ yield
\begin{equation}\label{pOmeg}
\partial_u \Omega= \frac{\bar{A}^{\prime}}{2A}p, \qquad \partial_v \Omega= 0\,,
\end{equation}
where we have used that $\partial_u F=0$, the definition of $K$ and that, by assumption, $q=0$. We infer from these equations that
\begin{equation}\label{eq.pvp}
\partial_v p = \frac{2A}{\bar A^\prime}\partial_{vu}\Omega=0\,,
\end{equation}
by the commutation of the derivatives $\partial_u$ and $\partial_v$.

Next, recalling that we defined the linear differential operator $D(\cdot)=\nn(\cdot)$, and $\nn = \partial_w-U\partial_u-V\partial_v$, we easily deduce the formula
\begin{equation}\label{pvD}
[\partial_v, D]=-\beta \partial_u-\partial_v V \partial_v\,.
\end{equation}
Additionally, Equations~\eqref{IICM1} and~\eqref{eq.pvp} yield
\begin{equation}\label{eq.om}
\frac{z\Omega}{X} + Dz = \text{ function of } w\,,
\end{equation}
and by taking its derivative $\partial_v$, and using $\partial_v z=\partial_v \Omega=0$ and Equation~\eqref{pvD}, we deduce that
$$\beta \partial_u z=0\,,$$
so that $\beta=0$ or $\partial_u z=0$ in $\mathcal W$.

Assume that $\beta \neq 0$, so that $\partial_u z=0$ in $\mathcal W$. By Equation~\eqref{ICM3},
$$z=-A\beta\,,$$
and therefore
$$Dz=-(A\beta)^\prime\,,$$
which implies, using Equation~\eqref{eq.om}, that $z\Omega/X$ is a function of $w$ only. Since
$$z=-A\beta \neq 0\,,$$
we deduce that $\Omega/X$ is a function of $w$ only, so in particular,
$$\partial_u \Omega = \frac{\Omega}{X}p\,.$$
Combining this formula with Equation~\eqref{pOmeg}, we obtain
\begin{equation*}
\Omega = \frac{\bar{A}^{\prime}}{2A}X\,.
\end{equation*}
Finally, it follows from Equation~\eqref{G1}, from $q=0$ and $z=-A\beta$, that
$$
 -\frac{p^2}{AX^2} = -\left(2 A \frac{K + \bar{A}^{\prime \prime} }{\bar{A}^{\prime}} + A^{\prime}\right) \Omega
-2 A \, D \Omega\,,
$$
so its $\partial_v$ derivative gives us
$$\partial_v D\Omega=0\,,$$
where we have used that $\partial_v p=0$ and $\partial_v \Omega=0$. Using again Equation~\eqref{pvD} we infer that
$$\beta \partial_u \Omega=0\,,$$
and since $\partial_u \Omega=\frac{\bar{A}^{\prime}}{2A}p\neq 0$ by a previous computation, we necessarily have that $\beta=0$, which is a contradiction. This completes the proof of the lemma.
\end{proof}

To complete this section, it remains to observe that computations analogous to those in the proof of Lemma~\ref{killingness}, allow us to show that
$$\xi := \gamma_2(w) \partial_v$$
is a Killing vector field (which commutes with $\uu$) in $\mathcal W$, where the function $\gamma_2$ is
$$\gamma_2(w) := \exp \left(\int _1^w \!\! \frac{K(t)+\bar A^{\prime\prime}(t)}{\bar A^\prime(t)}dt\right)\,,$$
which solves the ODE:
$$
\gamma_2^\prime = \frac{K + \bar{A}^{\prime \prime} }{\bar{A}^{\prime}}\gamma_2.
$$
Indeed, following the notation in the proof of Lemma~\ref{killingness}, in this case we also have $\xi(|E_i|^2)=0$, $i=1,2,3$, and
\begin{equation*}\label{commXI0}
\begin{split}
&[\xi, E_1]=[\xi, E_2]=0\,, \qquad [\xi, E_3] = - A X \beta \gamma_2 \, E_1
+ X\left(\gamma_2^\prime + \partial_v V \, \gamma_2 \right)E_2\, ,
\end{split}
\end{equation*}
which together with Lemma~\ref{L.q0}, Equation~\eqref{partialvy0} and the definition of $\gamma_2$, leads us to the conclusion that
\[
[\xi,E_3]=0\,.
\]
The rest of the proof is the same as in Lemma~\ref{killingness}.

\subsection{Step~6: final remarks to complete the proof}\label{SS.step6}
Summarizing, in the previous steps we have proved that the vector field $\uu$ commutes with a vector field $\xi$, cf. Lemma~\ref{killingness} and Section~\ref{SS.step5}, in the open region $\mathcal W$, and that $\xi$ is Killing. Moreover, since $\mathcal W$ is fibred by (regular) compact level sets of $b$ (and hence of $w$), the vector field $\xi$ is tangent to $\partial\mathcal W$. The only Killing fields of the Euclidean metric with compact orbits are the rotations, so we conclude that there are coordinates $(x_1,x_2,x_3)$ such that $\xi$ takes the form
\[
\xi = -x_2\partial_{x_1}+x_1\partial_{x_2}\,,
\]
up to a constant factor, in $\mathcal W$, but it extends naturally to the whole $\mathbb R^3$ (in fact, it is a general property that a local Killing field admits a unique global extension). By analyticity of both $\uu$ and $\xi$, we then infer that
\[
[\uu,\xi]=0
\]
in $\Omega$, thus showing that $\uu$ is an axisymmetric vector field. Moreover, since $\xi(w)=0$ in $\mathcal W$, then $\xi(b)=0$ in $\mathcal W$, and the analyticity of $b$ and $\xi$ implies that, in fact,
\[
\xi(b)=0
\]
in $\Omega$. In particular, $\xi$ is tangent to $\partial\Omega$ (each component is a regular level set of $b$), so we finally conclude that $\Omega$ is a rotationally symmetric domain. Moreover, the boundary of $\Omega$ consists of finitely many rotationally symmetric toroidal surfaces. Then, the poloidal section of $\Omega$, i.e., the intersection $\Omega \cap \{x_1=0\}$, is disjoint from the $x_3$-axis. In the next section we shall prove that this section is a convex disk if $\partial\Omega$ is connected, or a convex annulus otherwise.

\section{Convexity of the poloidal cross section}\label{S:polo}

In this final section we complete the proof of Theorem~\ref{T.main1} showing that the poloidal section of the rotationally symmetric domain $\Omega$ is either a convex disk or a convex annulus. We caution the reader that the notations used in this section are independent of Section~\ref{sec:proof}.

As argued in Section~\ref{SS.step6}, we can introduce Euclidean coordinates $(x_1,x_2,x_3)\in\mathbb R^3$, and the associated cylindrical coordinates $(r, \varphi, z)$, so that the analytic vector field $\uu$ is an axisymmetric solution to the stationary Euler equations in $\Omega$. Moreover, $\Omega$ is a rotationally symmetric domain whose poloidal section $\Sigma$ is contained in the half-plane $\{x_1=0\}\cap\{x_2>0\}$. It is well known that $\uu$ then reads as
\begin{equation*}
\uu=\frac{1}{r}\left(\partial_z \psi \, \partial_r + \frac{S(\psi)}{r} \partial_\varphi -\partial_r \psi \, \partial_z\right)\,,
\end{equation*}
where $S(\psi)$ is the \emph{swirl} and the analytic function $\psi(r,z)$ is the \emph{stream function} of $\uu$. Moreover, setting $\mathsf{x}:=r^2$, the function $\psi(\mathsf{x},z)$ satisfies the Grad-Shafranov equation
\begin{equation}\label{GSeq}
4\mathsf{x} \frac{\partial^2 \psi}{\partial \mathsf{x}^2}+\frac{\partial^2 \psi}{\partial z^2}+SS^\prime -\mathsf{x} b^\prime=0
\end{equation}
in $\Sigma$, which we parametrize using the coordinates $\mathsf x>0$ and $z\in\mathbb R$ (as usual, since the involved functions do not depend on the angular variable $\varphi$, we identify $\Omega$ with its poloidal section $\Sigma$). Here $b(\psi)$ is the Bernoulli function (which is expressed as a function of $\psi$). We remark that the function $\psi$ is analytic in $\Sigma$. The localizability condition in terms of $\psi$ reads as $|\uu|^2 = \mathcal{U}(\psi)^2$, for some function $\mathcal U$, and in coordinates this is equivalent to
\begin{equation}\label{ISODeq}
4\mathsf{x} \left(\frac{\partial \psi}{\partial \mathsf{x}}\right)^2+\left(\frac{\partial \psi}{\partial z}\right)^2 + S^2 = \mathcal{U}(\psi)^2 \mathsf{x}\,.
\end{equation}
Notice that this excludes the case that $\psi$ is constant everywhere in $\Omega$, and hence, by analyticity, most of the values of $\psi$ are regular. In what follows we shall work in a (two-dimensional) neighborhood $\mathcal W\subset\Sigma$ of an arbitrary regular level set
$$\{\psi(\mathsf{x},z)=c_0\}\,;$$
if $\mathcal W$ is narrow enough, it is clear that the vector field $\uu$ does not vanish at any point and $\psi$ has a nonvanishing gradient. In particular, the function $\mathcal U>0$ in $\mathcal W$.

In the following proposition we show that there is an ODE whose solutions give a parametrization of a regular level curve of the stream function $\psi(\xd,z)$ within the poloidal section. In particular, since each connected component of $\partial\Omega$ is a regular level set of $\psi$, the following proposition describes the geometry of each component of $\partial\Omega$.  It was first derived by Palumbo in~\cite{Pal68}, but we include here a detailed proof for the sake of completeness. We remark that, by analyticity, for each compact level curve of $\psi$, the set of points such that the vector tangent to the curve is vertical is finite.

\begin{proposition}
Around any point whose tangent vector is not vertical, the curve
$\{\psi(\mathsf{x},z)=c_0\}$  can be parametrized as a graph $z(\xd)$ that satisfies the ODE:
\begin{equation}\label{dydxpal}
4\left(\frac{dz}{d\xd}\right)^2 = \frac{(\xd+\lambda_1)^2}{\nu_1^2(\xd-\mu_1) - \xd(\xd + \lambda_1)^2}\,,
\end{equation}
where $\lambda_1$, $\nu_1$ and $\mu_1$ are constants that depend on the level set.
\end{proposition}

\begin{proof}
It is convenient to re-parametrize the function $\psi$ by introducing another function $\zeta=\zeta(\psi)$ that satisfies
$$
\frac{d\zeta}{d\psi}=\frac{1}{\mathcal{U}(\psi)}\,,
$$
which will simplify many computations later. With some abuse of notation, this allows us to write $\psi$ as a function of $\zeta$, i.e., $\psi=\psi(\zeta)$. We also define two functions $\mu(\psi), \nu(\psi)$ as
$$
-\frac{b^\prime}{\mathcal{U}}+\mathcal{U}^\prime  =: -\frac{2}{\nu}\,, \qquad
\frac{S^2}{\mathcal{U}^2} =: \mu\,,
$$
which we can understand as functions of $\zeta$ after composition with $\psi(\zeta)$.
It is also useful to set the following functions, where we understand $\zeta=\zeta(\psi(\xd,z))$:
\begin{equation*}
\fp:=\frac{\partial \zeta}{\partial \xd}\,, \quad \fq:=\frac{\partial \zeta}{\partial z}\,, \quad
\mathfrak{r}:=\frac{\partial\fp}{\partial \xd}\,, \quad  \mathfrak{t} := \frac{\partial \fq}{\partial z}\,.
\end{equation*}

After a straightforward computation, we obtain that Equations~\eqref{GSeq} and~\eqref{ISODeq} read as
\begin{equation}\label{GSeq+}
4\mathsf{x} \fr + \ft =\frac{2}{\nu}\xd - \frac{1}{2}\frac{d\mu}{d\zeta}\,, \qquad
4\xd \fp^2 + \fq^2 = \xd -\mu\,,
\end{equation}
where $\mu$ is considered as a function $\mu(\psi(\zeta))\equiv \mu(\zeta)$. Observe that the level set $\{\psi(\mathsf{x},z)=c_0\}$ can also be represented as a level set of the function $\zeta$, i.e.,
$$\{\zeta(\mathsf{x},z)=c_1\}\,.$$
In the neighborhood $\mathcal W$ let us introduce the change of variables
$$(\xd, z) \mapsto (\xd, \zeta(\xd,z))\,,$$
which allows us to simplify Equation~\eqref{GSeq+}:
\begin{equation}\label{GSeq++}
\frac{\partial \fp}{\partial \xd}=\frac{1}{2\nu(\zeta)}\,, \qquad
4\xd \fp^2 + \fq^2 = \xd -\mu(\zeta)\,.
\end{equation}
Here the unknowns are the functions $\fp(\xd,\zeta)$ and $\fq(\xd,\zeta)$. Integrating the first equation we get
$$\fp(\xd,\zeta)=\frac{1}{2\nu(\zeta)}(\xd + \lambda(\zeta))\,,$$
for some function $\lambda(\zeta)$. It is convenient to write the second equation as
$$\Big(\frac{\fq}{\fp}\Big)^2 = \frac{\xd -\mu(\zeta)}{\fp^2}-4\xd\,,$$
which is well defined almost everywhere in $\mathcal W$ (by analyticity). Introducing the previous equation into the right hand side of this expression, we obtain
\begin{equation}\label{eq.qp}
\frac14\Big(\frac{\fq}{\fp}\Big)^2=\frac{\nu^2(\xd-\mu)-\xd(\xd+\lambda)^2}{(\xd+\lambda)^2}\,.
\end{equation}

Consider the level set $\{\zeta(\xd,z)=c_1\}$, which is an analytic closed curve in the plane $(\xd,z)$. By the implicit function theorem, we can write this curve as a graph $z(\xd)$ around any point of the curve whose tangent vector is not vertical (this excludes finitely many points by analyticity). Since
\[
\frac{dz}{d\xd}=-\frac{\frac{\partial \zeta}{\partial \xd}}{\frac{\partial \zeta}{\partial z}}=-\frac{\fp}{\fq}\,,
\]
we infer from Equation~\eqref{eq.qp} that \eqref{dydxpal} holds true, where $\lambda_1:=\lambda(c_1)$, $\nu_1:=\nu(c_1)$ and $\mu_1:=\mu(c_1)$ are constants that depend on the level set $\{\zeta(\xd,z)=c_1\}$.
\end{proof}

Next we prove that the regular level curves of $\psi$ are convex (independently of the values of the defining constants).

For the solutions of Equation~\eqref{dydxpal} to exist, the range of $\xd$ should be the interval where the denominator $\mathcal{Q}(\xd):=\nu_1^2(\xd-\mu_1) - \xd(\xd + \lambda_1)^2$ is positive. But $\mathcal{Q}(0)<0$, so a necessary condition is that $\mathcal{Q}$ has 3 distinct real roots $\xd_{1} < 0 < \xd_{2} < \xd_{3}$, so that $\xd_{1}\xd_{2}\xd_{3}=-\nu_1 ^2 \mu_1<0$. For this, the discriminant of $\mathcal Q$ should be positive, that is
$$
\Delta_{\mathcal{Q}} = \nu_1 ^2\left(4 \nu_1 ^4-\nu_1 ^2 \left(8 \lambda_1 ^2+36 \lambda_1  \mu_1 +27 \mu_1 ^2\right)+4 \lambda_1 ^3 (\lambda_1 +\mu_1 )\right)>0
$$
and, if this is the case, the two positive roots of $\mathcal Q$ are:
\begin{align*}
&\xd_{2}=
\frac{2}{3}\left(\sqrt{\lambda_1^{2}+3\nu_1^{2}}
\cos(\tfrac{1}{3}\arccos(\Theta)-\tfrac{2\pi}{3})
- \lambda_1\right)\,,\\
&\xd_{3}
=
\tfrac{2}{3}\left(\sqrt{\lambda_1^{2}+3\nu_1^{2}}
\cos(\tfrac{1}{3}\arccos(\Theta))-\lambda_1\right)\,,
\end{align*}
where $\Theta:= \frac{2 \lambda_1 ^3-9 \nu_1 ^2 (2 \lambda_1 +3 \mu_1 )}{2 \left(3 \nu_1 ^2+\lambda_1 ^2\right)^{3/2}}$. Since the level set $\{\zeta=c_1\}$ is a closed curve, the point $\xd=-\lambda_1$, which corresponds to a horizontal tangent vector, must be in the interval $(\xd_{2},\xd_{3})$; in particular $\lambda_1<-\xd_2<0$.

In order to study the convexity of the curve, we compute its curvature, which is given by the expression
$$\kappa(\xd)=\frac{\frac{d^2 z}{d\xd^2}}{(1+(\frac{dz}{d\xd})^2)^{3/2}}\,.$$
For the upper half-branch of the curve, corresponding to $\xd \in (\xd_2, -\lambda_1]$, the coordinate $z$ grows with $\xd$, so the parametrization is $\frac{dz}{d\xd} = \frac{\xd + \lambda_1}{2 \sqrt{\QQ(\xd)}}$. The curvature of this part is given by
$$
\kappa(\xd)=2\, \frac{2 \QQ(\xd)- (\xd + \lambda_1) \QQ'(\xd)}{\left(4 \QQ(\xd)+(\xd +\lambda_1)^2\right)^{3/2}}\,.
$$
Let us notice that $\kappa(\xd_2)=\frac{2 \QQ'(\xd_2)}{(\xd_2 +\lambda_1)^2}$ is positive since $\QQ$ must be increasing at $\xd_2$. In fact, the numerator of $\kappa$ is an increasing function because
$$
\frac{d}{d\xd}(2 \QQ(\xd)- (\xd + \lambda_1) \QQ'(\xd))=\nu_1 ^2+3 (\xd+\lambda_1)^2 >0\,,
$$
thereby the positive sign of $\kappa$ is conserved on the interval $(\xd_2, -\lambda_1]$.
For the other upper half-branch of the curve, corresponding to $\xd \in [-\lambda_1, \xd_3)$, the coordinate $z$ decreases as $\xd$ grows, so the parametrization satisfies $\frac{dz}{d\xd} = -\frac{\xd + \lambda_1}{2 \sqrt{\QQ(\xd)}}$. The analysis is similar and we again obtain positive curvature. Since the lower branch of the curve must be considered with the reversed orientation, we conclude that the curvature of the closed curve $\{\zeta(\xd,z)=c_1\}$ is positive, for any value of $c_1$ (as far as the level set is regular). This implies that all the regular level curves are convex, and in particular we obtain the convexity of the poloidal section of each component of $\partial\Omega$. It is then clear that the only two-dimensional domains that are foliated (except for a nowhere dense set) by convex closed curves are a convex disk or a convex annulus, which implies the convexity of the poloidal cross section of the domain $\Omega$, as we wanted to prove.

\begin{remark}[Twist]
The rotation number of the Poincar\'e map  (associated to the poloidal cross-section) of the axisymmetric flow $\uu$ is given by
\begin{equation}\label{rotNB}
\rho(c)=\frac{S(c)}{2\pi}\int_0^{\ell(c)}\frac{1}{r(s; c) \abs{\nabla \psi}(r^2(s; c), z(s; c))}ds
\end{equation}
where $\ell(c)$ is the length of the closed meridian curve $\{\psi(r,z)=c\}$ with arclength element $ds$.
The field $\uu$ has \textit{no twist} if $\rho$ is a constant independent of $c$.
Using the localizability condition, along $\Gamma_c$ we have
$|\nabla \psi|(r(s; c),z(s; c)) =\sqrt{\mathcal{U}(c) r^2(s; c)-S^2(c)}$,
so Equation~\eqref{rotNB} becomes
\begin{equation*}
\rho(c)=\frac{S(c)}{2\pi}\int_0^{\ell(c)}\frac{1}{r(s; c) \sqrt{\mathcal{U}(c) r^2(s; c)-S^2(c)}}ds\,.
\end{equation*}
For a generic choice of the functions $S$ and $\mathcal U$, this quantity is not constant. For the particular case of Gavrilov's solution~\cite{Gavrilov}, its twist property has been established in detail in~\cite{bal}.
\end{remark}

\section*{Acknowledgements}

\noindent This work is supported by the grants CEX2023-001347-S and PID2022-136795NB-I00 (D.P.-S.) funded by MCIN/AEI /10.13039/501100011033.

\appendix
\section{Some remarks on Palumbo-Balzano's proof}\label{S:Balz}

In this appendix we briefly comment on the proof by Palumbo and Balzano in~\cite{PB86}. They claim that any toroidal isodynamic equilibrium is axisymmetric. The first naive observation is that they do not assume that the solutions are analytic $(C^\omega)$, but this is an important property that is tacitly used in several parts of their article (and we also assume it in our proof). This said, let us explain the most important issues and gaps of their proof:
\begin{enumerate}
\item A crucial and strong assumption in Palumbo-Balzano's proof is that they assume the magnetic field is ergodic on almost all the level surfaces of $b$. This immediately implies that any first integral of $\uu$ is functionally dependent with $b$, which considerably simplifies all the arguments. Dropping this unjustified assumption makes the proof much more complicated, see Step~3 in Section~\ref{sec:proof}, where we deal with the general case.
\item Palumbo and Balzano do not construct a Killing symmetry of $\uu$, as we do in Steps~4 and~5 of our proof. Instead, they use the assumed ergodicity of $\uu$ and the theory of curves in $\mathbb R^3$ to show that the level sets of certain function are circles on each toroidal surface. They continue to argue that the centers of these circles belong to the same axis because they are orthogonal to the same pencil of planes. In this way they conclude that the level sets of $b$ are surfaces of revolution. This proof looks incomplete to us, and we have not been able to check its correctness. Moreover, this would only prove that the level sets of the function $b$ are rotationally symmetric, but it remains to prove that $\uu$ is axisymmetric itself (for which one needs, apparently, a different article by Palumbo~\cite{Pal86}). Our proof is more concise, filling all these issues and establishing symmetry of all the objects at the same time, using the same tool, i.e., a Killing vector field.
\item They construct a local coordinate system $(u,v,w)$ in a qualitative way (using geometric considerations), but the construction is not precise and no commutation relations are checked. In particular, the relation between the frames $\{\uu,\uu\times \nabla b\}$ and $\{\partial_u,\partial_v\}$ is not established, as well as several other important identities. See Step~1 in Section~\ref{sec:proof} for a rigorous proof.
\item There is no clear distinction between local and global objects, and no proof that all the relevant functions and vector fields that are used in their proof admit a coordinate free global expression. See also Step~1 in our proof.
\item It is fair to acknowledge that Palumbo and Balzano also exploited the fact that the metric is Euclidean, but instead of writing the zero curvature equations of the metric (as we do in Step~2 of Section~\ref{sec:proof}) they work with the intrinsic geometry of the level sets of coordinate surfaces (they use the Gauss and Codazzi-Mainardi equations of surface geometry). This makes the proof more complicated and conceptually less clear.
\item They do not consider the degenerate case that $q=0$ everywhere, so their proof tacitly assumes that this cannot happen. We cover this case, which is highly non trivial, in Step~5 of our proof.
\end{enumerate}

\bibliographystyle{amsplain}

\end{document}